\documentclass{amsart}
\usepackage{FuLoSptoSt,amssymb}
\datver{0.9A; Revised: 30-8-2013}
\begin{document}
\title[Fusion loop-spin structures]
{Equivalence of string and fusion loop-spin structures}

\author{Chris Kottke}
\address{Department of Mathematics, Northeastern University}
\email{c.kottke@neu.edu}
\author{Richard Melrose}
\address{Department of Mathematics, Massachusetts Institute of Technology}
\email{rbm@math.mit.edu}
\begin{abstract}
The importance of the fusion relation of loops was recognized in the
context of spin structures on the loop space by Stolz and Teichner and
further developed by Waldorf. On a spin manifold $M$ the equivalence
classes of `fusive' spin structures on the loop space $\cL M,$
incorporating the fusion property, strong regularity and
reparameterization-invariance, are shown to be in 1-1 correspondence with
equivalence classes of string structures on $M.$ The identification is
through the affine space of `string' cohomology classes considered by
Redden.  \end{abstract}

\maketitle

\section*{Introduction}

In \cite{Stolz-Teichner2005} Stolz and Teichner observed the importance of
the notion of a fusion structure on objects over the loop space of a
manifold arising from transgression constructions (it is implicit in
Barrett \cite{barrett1991holonomy}). If $M$ is an oriented manifold they
exhibit an isomorphism between fusion orientations of the looped frame
bundle over the loop space, $\cL M=\CI(\bbS;M),$ and spin structures on
$M,$ so refining and extending earlier results of McLaughlin
\cite{McLaughlin92} making precise an observation of Atiyah
\cite{MR87h:58206}. Using the treatment of bundle 2-gerbes by Stevenson
\cite{Stevenson2001} and categorical constructions, Waldorf
\cite{Waldorf2012} has shown that the existence of a fusion loop-spin
structure, i.e.\ a spin structure on the loop manifold, is equivalent to
the existence of a string structure on the manifold. Redden
\cite{Redden2011} has identified string structures, up to equivalence, with
certain integral cohomology classes on the spin-oriented frame bundle. Here
we develop direct forms of Waldorf's transgression and regression
constructions, the former via holonomy of projective unitary bundles and
the latter via bundle gerbes in the sense of Murray \cite{Murray1} with a
strong emphasis on the smoothness of these objects. In particular it is
shown that `fusive loop-spin structures', which are fusion spin structures
on the loop space with additional regularity and equivariance properties
under reparameterization, and string structures, both up to natural
equivalence, are in 1-1 correspondence with the cohomology classes
introduced by Redden, denoted by $\strC(F)$ below. The existence of such
loop spin structures which are equivariant under reparameterization by
diffeomorphisms of the circle answers a long-standing question, posed for
instance by Brylinski \cite{Brylinski3} in relation to elliptic
cohomology. It is expected that the smooth loop-spin structures constructed
here can serve as the basis of an analytic discussion of the Dirac-Ramond
operator.

If $M$ is a spin manifold of dimension $n\ge5$ with $\pi_F:F\longrightarrow
M$ the spin-oriented frame bundle, there is an exact sequence (compare
Redden \cite{Redden2011})
\begin{equation}
\xymatrix{
0\ar[r]&H^3(M;\bbZ)\ar[r]^-{\pi^*_F}& H^3(F;\bbZ)\ar[r]^-{i^*_\fib}&H^3(\Spin;\bbZ)\ar[r]&\ha
p_1\cdot\bbZ\ar[r]&0
}
\label{FuLoSptoSt.28}\end{equation}
where the second map is restriction to the fibre and the third,
transgression, map arises from the construction of the spin-Pontryagin
class $\ha p_1\in H^4(M;\bbZ)$ associated to $F.$ The cohomology classes of interest
are those in the coset
\begin{equation}
\strC(F)=\{\sigma \in H^3(F;\bbZ);i^*_\fib(\sigma)=\tau\}
\label{FuLoSptoSt.29}\end{equation}
of $H^3(M;\bbZ)$ in $H^3(F;\bbZ)$, where $\tau$ is a chosen generator of
$H^3(\Spin;\bbZ)=\bbZ.$ 

A string group in dimension $n$ gives a short exact sequence of topological
groups
\begin{equation}
\xymatrix{
K\ar[r]&\String\ar[r]&\Spin
}
\label{FuLoSptoSt.30}\end{equation}
where $K$ is a $K(\bbZ;2)$ and $\pi_3(\String)=\{0\}.$ A string structure
on $M$ is a principal bundle for $\String$ refining $F:$  
\begin{equation}
\xymatrix{
\String\ar[d]\ar[r]& F_{\operatorname{Str}}\ar[d]\\
\Spin\ar[r]&F\ar[d]^{\pi_F}\\
&M.
}
\label{FuLoSptoSt.31}\end{equation}
In \cite{Redden2011}, Redden shows that isomorphism classes of string structures
are in bijection with classes in $\strC(F).$

We refer to \S\ref{SectFuLoSp} for the precise description of a `fusive
loop-spin structure', but in brief it is a lifting of the principal $\cL
\Spin$-bundle $\cL F$ over $\cL M$ to a principal bundle with structure
group the basic central extension of $\cL \Spin$ with additional
multiplicative structure corresponding to the join of paths to form loops
(i.e.\ fusion), strong equivariance under change of parameterization of
loops and strong smoothness properties.

\begin{thmstar} For a spin manifold $M^n,$ $n\ge5,$ there is a
bijection from the set of equivalence classes of fusive loop-spin
structures on $M$ to $\strC(F),$ as $H^3(M;\bbZ)$ torsors, given by the
Dixmier-Douady class of the associated bundle gerbe on $F.$ 
\end{thmstar}

Note that $\strC(F)$ is non-empty if and only if $\ha p_1=0$ on $M;$ that this is
equivalent to the existence of a string structure on $M$ and also to the
existence of a fusion loop-spin structure was shown by Waldorf \cite{Waldorf2012}.

We do very little here with regard to string structures on $M,$ relying on
the work of Stolz, Teichner and Redden. The main novelty of our results is
therefore the existence of a `lithe' and $\Dff(\bbS)$-equivariant, fusion
loop-spin structure corresponding to each class in $\strC(F),$ where
`lithe' is a strong smoothness condition, similar to what was called
`super-smooth' by Brylinski in \cite{Brylinski3}.  Combined with the
existence results on string structures this realizes the equivalence
between a strengthened version of smooth fusion loop-spin structures and
smooth string structures as envisaged in \cite{Waldorf2012}. The main
reason for our interest in this isomorphism is that it is a first step
towards an ongoing careful examination of the Dirac-Ramond operator on the
loop manifold, associated to a string structure, with the hope of further
elucidating the index theorem of Witten \cite{Witten-loop} and ideas of
Brylinski \cite{Brylinski90} and Segal \cite{Segal-Bourbaki}.

Before approaching the proof of the Theorem above, we first recall in
\S\ref{Sect.loop-orient} the identification by Stolz and Teichner of spin
structures on an oriented manifold with orientations of the loop space
(loop-orientations) satisfying the fusion condition. Fusive functions on
the loop space valued in $\UU(1)$ are analysed in \S\ref{SectFuMa} and
shown to give a refinement of the first integral cohomology group making it
isomorphic, by regression, to the second cohomology group of the
manifold. This is extended to a corresponding result for fusive circle
bundles in \S\ref{SectFuCiBu} where a bundle gerbe construction is used to
show that these are classified up to equivalence by integral 3-cohomology
on the manifold. The classification of string structures, essentially
following the work of Redden \cite{Redden2011}, is recalled in
\S\ref{SectSG}. Fusive loop-spin structures, are defined in
\S\ref{SectFuLoSp} and equivalence classes of them are shown to map to the
integral 3-classes representing string structures. The proof of the Theorem
above is completed in \S\ref{Constr} by constructing a loop-spin structure
corresponding to each string class, using the discussion in \S\ref{RegCF}
of circles bundle over the loop space associated each such class and the
description in \S\ref{Blipping} of the blip map, used to construct the
extension of the loop spin action.  The three appendices contain material
on the notion of lithe regularity introduced here and then applied to loop
manifolds and on properties of the basic central extension of the loop
group of $\Spin.$

\paperbody
\section{Spin and loop-orientation}\label{Sect.loop-orient}

To serve as a model for our results below, we begin with a proof of the
theorem of Stolz and Teichner \cite{Stolz-Teichner2005} showing the
equivalence between spin structures and fusion orientations of the loop
space for a compact oriented manifold $M.$ We relate this to the
`$\bbZ_2$-fusive 1-cohomology' of the loop space, a notion which is
generalized substantially below.

Giving $M$ a Riemann metric reduces the oriented frame bundle to the
principal $\SO=\SO(n)$ bundle of oriented orthonormal frames, denoted
$F_{\SO}.$ Since $\pi_1(\SO)=\bbZ_2$ the loop space $\cL\SO$ has two
components and an orientation on it is a reduction of the structure group
$\cL\SO$ of $\cL F_{\SO}$ to the component of the identity, which is
naturally
\begin{equation}
\cL\Spin\longhookrightarrow \cL \SO.
\label{FuLoSptoSt.148}\end{equation}
Such orientations may be identified with the continuous maps
\begin{equation}
o:\cL F_{\SO}\longrightarrow \bbZ_2
\label{FuLoSptoSt.149}\end{equation}
taking both signs over each loop in $\cL M;$ the orientation subbundle is then
$\{o=+1\}.$ The smooth loop space is dense in the energy loop space $\cLE
F_{\SO}=H^1(\bbS;F_{\SO})$ and the latter retracts onto the former, so each orientation
extends uniquely to a continuous function on $\cLE F_{\SO}.$ 

Without the condition on the fibers of $\cL F_{\SO},$ one can simply consider the
group of maps \eqref{FuLoSptoSt.149} for any finite dimensional manifold
$Z.$ The smooth and energy path spaces $\cI Z =\CI([0,2\pi];Z),$ $\cIE
Z=H^1([0,2\pi];Z)$ both fiber over $Z^2;$ as discussed in
\S\ref{Sect.Loops}. The $k$-fold fiber product with respect to this
fibration, $\cI Z^{[k]},$ consists of the $k$-tuples of paths having the
same endpoints. Then the basic `fusion map' is
\begin{equation}
\begin{gathered}
\fusm : \cIE^{[2]}Z\longrightarrow \cLE Z \\
\fusm(\gamma_1,\gamma_2) = l,\ l(t) =
\begin{cases}
\gamma_1(2t), &0 \leq t \leq \pi, \\
\gamma_2(4\pi - 2t), & \pi \leq t \leq 2\pi
\end{cases}
\end{gathered}
	\label{E:fusion_map}
\end{equation}
taking a pair of paths with the same endpoints to the loop obtained by
joining the first path with the reverse of the second. In general the fusion of
smooth paths is not a smooth loop and the image is precisely the space of
piecewise smooth loops within the energy space, with possible discontinuities in derivatives at
$\set{0,\pi}\in\bbS.$

The simplicial projections $\pi_{ij}:\cIE^{[3]}Z\longrightarrow \cIE^{[2]}Z,$
obtained by keeping the $i$th and $j$th entries for $ij=12,$ $23$ and $13,$
generate three lifted maps 
\begin{equation*}
\fusm_{ij}=\fusm\circ\pi_{ij} : \cIE^{[3]}Z \to \cLE Z
\label{FuLoSptoSt.405}\end{equation*}
and then additional {\em fusion condition} on a map $o : \cLE Z \to \bbZ_2$ is
the simplicial multiplicativity property
\begin{equation}
\fusm_{12}^*o\cdot \fusm_{23}^*o=\fusm_{13}^*o\Mon \cI^{[3]}Z.
\label{FuLoSptoSt.150}\end{equation}
In other words, for every triple of paths $(\gamma_1,\gamma_2,\gamma_3) \in
\cIE^{[3]} Z$ with the same endpoints, the function satisfies
$o(l_{12})o(l_{23}) = o(l_{13})$, where $l_{ij} = \psi(\gamma_i,\gamma_j)$ (see
Figure~\ref{F:fusion}).

\begin{figure}[t]
\begin{tikzpicture}[rotate=-30]
\draw (-2,-2)  .. controls (-1.5,-1.7) and (1.2,-0.8)  .. node[below] {$\gamma_2$} (0.2,0.2) 
	.. controls (-0.3,0.7) and (-0.7,0.3)  .. (-0.2,-0.2)
	.. controls (0.8,-1.2) and (1.5, 1.5).. (2,2);
\draw (-2,-2)  .. controls (-1.5,-1) and (-1,0) .. (-1,1) node[above left] {$\gamma_3$} .. controls (-1,2) and (1,3).. (1,2) .. controls (1,1) and (2,2).. (2,2);
\draw (-2,-2)  .. controls (-2,-2) and (0.2,-2.2) .. (1.2,-1.2) node[below] {$\gamma_1$} .. controls (2.2,-0.2) and (2,2).. (2,2);
\end{tikzpicture}
\caption{Fusion of three paths.}
\label{F:fusion}
\end{figure}
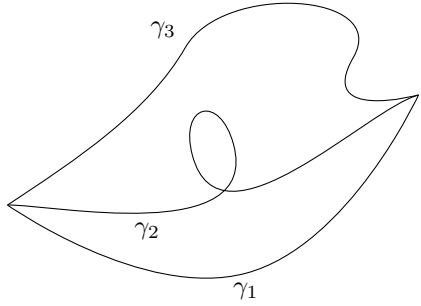

Since $o$ is locally constant, there is no need to consider any refined
notions of smoothness. For the same reason such a map is
automatically invariant under reparameterization.

\begin{proposition}\label{FuLoSptoSt.153} For any connected compact manifold
the group of fusive maps,
\begin{equation*}
H^0_{\fus}(\cL Z;\bbZ_2)=
\left\{\text{continuous }o:\cL Z\longrightarrow
\bbZ_2\text{ satisfying \eqref{FuLoSptoSt.150}}\right\},
\label{FuLoSptoSt.154}
\end{equation*}
is naturally isomorphic to $H^1(Z;\bbZ_2).$
\end{proposition}

\begin{proof} 
Extending by continuity to the energy space $\cLE Z$ and pulling back by the
fusion map, such a map $o$ defines

\begin{equation}
\tilde o:\cIE^{[2]}Z\longrightarrow \bbZ_2,\ \tilde o=o\circ\fusm.
\label{FuLoSptoSt.255}\end{equation}
The fusion condition implies that the relation on the trivial $\bbZ_2$ bundle
over $\cIE Z:$
\[
\cIE Z \times \bbZ_2 \ni (\gamma_1,e_1) \sim (\gamma_2,e_2) \Longleftrightarrow
	(\gamma_1,\gamma_2) \in \cIE^{[2]}Z,\ e_1 = \tilde o(\gamma_1,\gamma_2) e_2
\label{FuLoSptoSt.362}
\]
is an equivalence relation over the fibers of $\cIE Z\longrightarrow Z^2$. Thus,
this trivial bundle descends to a $\bbZ_2$ bundle $N$ over $Z^2$, determined up
to isomorphism by a cohomology class in $H^1(Z^2; \bbZ_2) \cong H^1(Z; \bbZ_2)
\oplus H^1(Z; \bbZ_2)$. However since $\cI Z$ retracts to the constant paths
under path contraction to the midpoint, identified with the subspace $Z \subset
\cI Z$ embedded over the diagonal $\Diag \subset Z^2$, it follows that the
pull-back of $N$ to the diagonal is trivial and then it follows that $N$ is
isomorphic to $L^{-1} \boxtimes L = \pi_1^*L^{-1} \otimes \pi_2^* L$ for a
$\bbZ_2$-bundle $L \longrightarrow Z$ since its cohomology class is of the form
$\pi_2^*\alpha -\pi_1^*\alpha$ for a unique $\alpha \in H^1(Z;\bbZ_2).$ 

If the class $\alpha$ vanishes then $L$ is trivial, and pulling back a global
section to $N \cong L^{-1}\boxtimes L$ and then up to $\cI Z$ gives a
continuous function $o' : \cI Z \to \bbZ_2$ such that $\tilde
o(\gamma_1,\gamma_2) = {o'}^{-1}(\gamma_1){o'}(\gamma_2).$ However since $\cI Z$ is connected
the only such functions are $o' = \pm 1$ and then $o \equiv 1.$
This fixes an injective `regression' map
\begin{equation}
\Rg:H^0_\fus(\cL Z;\bbZ_2)\longrightarrow H^1(Z;\bbZ_2).
\label{FuLoSptoSt.363}\end{equation}

To complete the proof of the Proposition it suffices to construct an
inverse `enhanced transgression' map. Given a class in $H^1(Z; \bbZ_2)$ 
represented by a $\bbZ_2$-bundle $L \longrightarrow Z$, there is a `holonomy' map
\begin{equation}
L\times_{\ev(0)}\cL Z\longmapsto \bbZ_2
\label{FuLoSptoSt.364}\end{equation}
defined on $(e,l) \in L\times_{\ev(0)} \cL Z$ by extending the initial
value $e \in L_{l(0)}$ to a continuous section of $L$ over the path
$[0,2\pi]\longrightarrow Z$ covering $l$ and taking the difference of the endpoints in
$L_{l(2\pi)} = L_{l(0)}.$ This map is invariant under the $\bbZ_2$ action
on $L$ so descends to a map $h : \cL Z\longrightarrow \bbZ_2$ which satisfies the
fusion condition. The same holonomy map results from isomorphic bundles so
the assignment of $L$ to $h$ descends to a map
\begin{equation}
\Tg_\fus:H^1(Z;\bbZ_2)\longrightarrow H^0_{\fus}(\cL Z;\bbZ_2).
\label{FuLoSptoSt.365}\end{equation}

That regression and fusive transgression are inverses follows from the fact
that the transgression of $L$ followed by regression is represented by the
class of the bundle $L^{-1} \boxtimes L$ over $Z^2.$ Indeed, using the fact
that $L = L^{-1}$ here due to the structure group, the pull-back of
$L^{-1}\boxtimes L$ to $\cI Z$ is trivialized by taking constant sections.
These trivializations of the fibers at $\gamma_1$ and $\gamma_2$, where
$(\gamma_1,\gamma_2)\in\cI^{[2]}Z$ have the same end-points, are identified by
the holonomy of $L.$ Thus $\Rg\circ\Tg_\fus=\Id$ and both are isomorphisms.
\end{proof}

Recall that a spin structure on $M$ is a lift of $F_{\SO}$ to a principal $\Spin=\Spin(n)$
bundle which we denote simply as $F,$ since it is a central object below,
giving a commutative diagram
\begin{equation}
\xymatrix{
\Spin\ar@{-}[r]\ar[d]& F\ar[d]\\
\SO\ar@{-}[r]&F_{\SO}\ar[d]\\
&M.
}
\label{FuLoSptoSt.151}\end{equation}
There is an exact sequence
\[
0 \longrightarrow H^1(M; \bbZ_2) \longrightarrow H^1(F_{\SO}; \bbZ_2)
\longrightarrow H^1(\SO; \bbZ_2) \equiv \bbZ_2 \longrightarrow w_2(M)\cdot \bbZ_2
\longrightarrow 0
\]
obtained for instance from the $E_2$ page of the Leray-Serre spectral
sequence for the fibration $F_{\SO}\longrightarrow M.$ Spin structures on
$M$ exist if and only if $w_2(M) = 0,$ and then their equivalence classes
are in bijective correspondence with classes in $H^1(F_{\SO}; \bbZ_2)$
which restrict on fibers to generate $H^1(\SO; \bbZ_2)$; equivalently, the
classes form a torsor over $H^1(M; \bbZ_2).$

\begin{theorem}[Stolz and Teichner \cite{Stolz-Teichner2005}]%
\label{FuLoSptoSt.152} For an oriented compact manifold of dimension
$n>5,$ fusion orientation structures are in 1-1 correspondence with spin structures.
\end{theorem}

\begin{proof} The fusion orientations \eqref{FuLoSptoSt.149} form a torsor
over $H^0_{\fus}(\cL M;\bbZ_2) \cong H^1(M; \bbZ_2).$ Moreover,
Proposition~\ref{FuLoSptoSt.153} associates to each of them a $\bbZ_2$
bundle, $L,$ over $F_{\SO}.$ The special property of $o,$ as opposed to a
general element of $H^0_{\fus}(\cL F_{\SO};\bbZ_2),$ is that it takes both
signs over each element of $\cL M.$ By restriction this translates to the
condition that $L$ is non-trivial as a bundle over each fiber of $F_{\SO}$
which in turn means it represents a $\Spin$ double cover, and that the
corresponding class in $H^1(F_{\SO};\bbZ_2)$ restricts to a generator of
$H^1(\SO; \bbZ_2)$.  Conversely, a spin structure is just such a $\bbZ_2$
bundle over $F_{\SO}$ which is non-trivial over each fiber and which
therefore has holonomy which is non-trival over fiber loops above each loop
in $M.$
\end{proof}

\section{Fusive functions} \label{SectFuMa}

Next we pass from maps into $\bbZ_2$ as in \eqref{FuLoSptoSt.154} to maps 
\begin{equation}
f:\cLE M\longrightarrow \UU(1).
\label{FuLoSptoSt.155}\end{equation}
satisfying the fusion condition. We will distinguish between the three
classes of `fusion,' `fusion with figure-of-eight' and `fusive' functions
with increasing restrictions. The third class is closely related to the
functions on the quotient of the loop space by `thin homotopy equivalence,'
considered by Waldorf, \cite{waldorf2009transgressionI}, since this
equivalence implies reparametrization invariance. Variants of the `regression' and
`enhanced transgression' maps, but using the pointed path space, are also
contained in Waldorf's work.
 
First we work in the topological category.

\begin{proposition}\label{FuLoSptoSt.448} The group $\cC_{\fsn}(\cL M;\UU(1))$ of
continous functions \eqref{FuLoSptoSt.155} which satisfy the fusion condition 
\begin{equation}
\begin{gathered}
\fusm_{12}^*f\cdot\fusm_{23}^*f=\fusm_{13}^*f\Mon\cIE^{[3]}M\\
\pi_{ij}:\cI^{[3]}M\longrightarrow \cIE^{[2]}M,\
\fusm:\cIE^{[2]}M\longrightarrow \cLE M\text{ in \eqref{E:fusion_map}}
\end{gathered}
\label{FuLoSptoSt.450}\end{equation}
has path components isomorphic to
\begin{equation}
\{\beta \in H^2(M^2;\bbZ);\beta\big|_{\Diag}=0\} \subset H^2(M^2; \bbZ).
\label{FuLoSptoSt.451}\end{equation}
\end{proposition}

\begin{proof} As already observed by Barrett, \cite{barrett1991holonomy}, a continous
function \eqref{FuLoSptoSt.155} defines a relation on $\cIE M\times\UU(1),$  
\begin{equation*}
(\gamma_1,z_1)\sim(\gamma_2,z_2)\Longleftrightarrow (\gamma_1,\gamma_2)\in
\cIE^{[2]}M,\ z_1=f(\gamma_1,\gamma_2)z_2
\label{FuLoSptoSt.452}\end{equation*}
with the fusion condition \eqref{FuLoSptoSt.450} equivalent to this being
an equivalence relation. Thus each $f$ determines a circle bundle $L_f$
over $M^2$ and hence there is a map to the Chern class $\beta =[L_f]\in
H^2(M^2;\bbZ).$ By construction, $L_f$ lifts to the trivial bundle over
$\cIE M.$ This space retracts onto the constant paths, identified with
$M\longrightarrow\Diag\hookrightarrow M^2$ embedding as the diagonal. It
follows that $L_f$ is trivial over $\Diag$ and hence the Chern class lies
in the subgroup \eqref{FuLoSptoSt.451}.

If two functions are homotopic through fusion functions then the
corresponding bundles are isomorphic and hence the map to Chern classes
descends to the group of components. Furthermore the product of functions,
which descends to a product on path components, maps to the tensor product
of bundles so the map is actually a homomorphism of groups.

To see that this map is surjective, observe that a circle bundle over $M^2$
with class in \eqref{FuLoSptoSt.451} is trivial when lifted to $\cIE M$ by
the end-point fibration since it is isomorphic to its pull-back from the
diagonal as constant paths. Choosing a global section $g$ identifies it
with $\cIE M \times \UU(1)$ and taking the difference $g(\gamma_1) =
f(\gamma_1,\gamma_2)g(\gamma_2)$ between the values of this section at
pairs $(\gamma_1,\gamma_2) \in \cIE^{[2]} M \cong \cLE M$ defines a
continous function \eqref{FuLoSptoSt.155} satsifying the fusion condition,
and which generates the original circle bundle.

To see injectivity, suppose $f$ generates a trivial bundle $L_f \longrightarrow
M^2$.  Since $L_f = \cIE M \times \UU(1)/\sim_f$, The pull-back to $\cIE M$ of
a global section of the latter may be regarded as a continuous function
$g : \cIE M \longrightarrow \UU(1)$ which satisfies
\begin{equation}
f(\gamma_1,\gamma_2)=g(\gamma_1)g^{-1}(\gamma_2).
\label{FuLoSptoSt.454}\end{equation}
Pulling $g$ back under the retraction of $\cIE M$ to $M$ gives an homotopy
$f_t$ through fusion functions with end-point the constant function $1.$
\end{proof}

We proceed to identify a group of functions on $\cL M$ which are classified
up to homotopy by the degree 2 cohomology of $M.$ For this we require a
second multiplicative condition.

Recall that the join operation on paths is a map
\begin{equation}
j : \pi_{12}^* \cIE M \times_{M^3} \pi_{23}^* \cIE M \longrightarrow \pi_{13}^* \cIE M
\label{E:join}
\end{equation}
where $\pi_{ij} : M^3 \longrightarrow M^2$ are the projections. Pulling back the fiber
products $\cI^{[2]} M \longrightarrow M^2$, the join defines a map
\[
j^{[2]} : \pi_{12}^* \cIE^{[2]} M \times_{M^3} \pi_{23}^* \cIE^{[2]} M
\longrightarrow \pi_{13}^* \cIE^{[2]} M
\]
which, along with the fusion map $\psi : \cIE^{[2]}M \longrightarrow \cLE M$ generates
the {\em figure-of-eight product} on loops:
\begin{equation}
\begin{gathered}
	J : \cLE M \times_{\ev(0) = \ev(\pi)} \cLE M \longrightarrow \cLE M \\
	(l_1,l_2) = \big(\psi(\gamma_1,\gamma_2),\psi(\gamma'_1,\gamma_2')\big) 
	\mapsto \psi\big(j(\gamma_1,\gamma'_1),j(\gamma_2,\gamma'_2)\big);
\end{gathered}
	\label{E:figure_of_eight}
\end{equation}
see Figure~\ref{F:fo8}.
\begin{figure}[t]
\begin{tikzpicture}[rotate=-25]
\draw (-2,-2) .. controls (-1,-2) and (0,-1) .. node[below] {$\gamma_1$} (0,0) .. controls (1,0) and (2,1) .. node[below] {$\gamma'_1$} (2,2);
\draw (-2,-2) .. controls (-2,-1) and (-1,0).. node[above] {$\gamma_2$} (0,0) .. controls (0,1) and (1,2) .. node[above] {$\gamma'_2$} (2,2);
\end{tikzpicture}
\caption{The figure-of-eight product on loops.}
\label{F:fo8}
\end{figure}
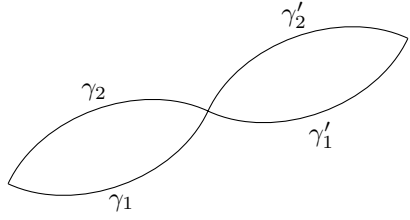

\begin{definition}\label{FuLoSptoSt.407} A {\em fusion-figure-of-eight function}
is a continuous map \eqref{FuLoSptoSt.155} which satisfies the fusion condition
\eqref{FuLoSptoSt.450} and is multiplicative under the figure-of-eight product:
\begin{equation}
	\Pi_1^* f\,\Pi_2^* f = J^* f \Mon \cLE M \times_{\ev(0) = \ev(\pi)} \cLE M.
\label{FuLoSptoSt.431}\end{equation}
The group of such functions is denoted $\cC_{\foe}(\cL
M;\UU(1))\subset\cC_{\fsn}(\cL M;\UU(1)).$ 
\end{definition}

We define the {\em fusion 1-cohomology} of $\cL M$ to be
\begin{equation}
H^1_{\fus}(\cL M)=\cC_{\foe}(\cL M;\UU(1))/\text{Fusion homotopy.}
\label{FuLoSptoSt.210}\end{equation}
Equivalence here is homotopy in $\cC_{\fsn}(\cL M;\UU(1));$ we show below
that this is equivalent to homotopy in $\cC_{\foe}(\cL M;\UU(1)).$
There is a natural map of abelian groups obtained by ignoring the fusion conditions
\begin{equation}
H^1_{\fus}(\cL M)\longrightarrow H^1(\cL M;\bbZ),
\label{FuLoSptoSt.212}\end{equation}
which is partial justification for our choice of degree in the notation. It is
shown in Theorem~\ref{H1fus} that fusion 1-cohomology in this sense does realize the
idea of Stolz and Teichner that the fusion condition `promotes'
transgression to an isomorphism.
Proposition~\ref{FuLoSptoSt.448} now leads to the regression map from fusive
1-cohomology to 2-cohomology on $M.$

\begin{proposition}\label{FuLoSptoSt.410} Associated to each element of
$\cC_{\foe}(\cL M;\UU(1))$ is a circle bundle over $M^2$ of the form $L^{-1}\boxtimes L$ 
for $L \longrightarrow M,$ the Chern class of which induces an injective homomorphism 
\begin{equation}
\Rg:H^1_{\fus}(\cL M)\longrightarrow H^2(M;\bbZ).
\label{FuLoSptoSt.411}\end{equation}
\end{proposition}
\noindent This is shown to be an isomorphism in Theorem~\ref{H1fus} below.

\begin{proof} Consider the circle bundle $L_f$ constructed in the proof of
Proposition~\ref{FuLoSptoSt.448}, from a function in $\cC_{\foe}(\cL M; \UU(1))$.
The second multiplicativity condition \eqref{FuLoSptoSt.431} gives an isomorphism
\begin{equation}
\theta=\theta_f:\pi_{12}^*L_f\otimes\pi_{23}^*L_f\longrightarrow \pi_{13}^*L_f\Mover
M^3.
\label{FuLoSptoSt.412}\end{equation}
so the Chern class $\alpha(f) \in H^2(M^2; \bbZ)$ satisfies 
\begin{equation}
\pi_{23}^* \alpha(f) - \pi_{13}^* \alpha(f) + \pi_{12}^* \alpha(f) = 0 \in H^2(M^3; \bbZ).
	\label{E:cohom_trivial_M3}
\end{equation}
Choose a point $\bar m \in M$ and consider the embeddings
\[
\begin{aligned}
	i_2 &: M^2 \longhookrightarrow \bar m \times M^2 \subset M^3 \\
	i_1 &: M \longhookrightarrow \bar m \times M \subset M^2.
\end{aligned}
\]
Pulling back \eqref{E:cohom_trivial_M3} to $M^2$ by $i_2^*$ and using the fact that
$\pi_{12} \circ i_2 = i_1 \circ \pi_1$ and $\pi_{13} \circ i_2 = i_1 \circ
\pi_2$ as maps $M^2 \longrightarrow \bar m \times M \subset M^2$ while
$\pi_{23} \circ i_2 = \Id,$ it follows that
\[
	\alpha(f) = \pi_2^*\beta(f) - \pi_1^* \beta(f),\
	\beta(f) = i_1^* \alpha(f) \in H^2(M; \bbZ).
\]
Proposition~\ref{FuLoSptoSt.448} shows that the regression map is well-defined
by
\begin{equation}
\Rg(f)=\beta(f)\Mon H^1_{\fus}(\cL M)
\label{FuLoSptoSt.413}\end{equation}
with the equivalence relation being homotopy in $\cC_{\fsn}(\cL M;\UU(1)).$
\end{proof}

Suppose $f \in \cC_{\foe}(\cL M ; \UU(1))$ satisfies $\Rg(f) = 0,$ so that
the associated bundle $L_f \cong L^{-1}\boxtimes L$ with $L$ trivial.
Taking a global section of $L$ and pulling the induced global section of
$L_f$ back to $\cIE M$ as in the proof of Proposition~\ref{FuLoSptoSt.448}
shows that there is continuous function $e : \cIE M \longrightarrow \UU(1)$
such that
\begin{equation}
\begin{gathered}
	f(\gamma_1,\gamma_2) = e(\gamma_1)^{-1}e(\gamma_2), \\
	e(r\gamma) = e(\gamma)^{-1}, \\
	e\big(j(\gamma_1,\gamma_2)\big) = e(\gamma_1)e(\gamma_2) 
\end{gathered}
	\label{E:ffoe_simp_triv}
\end{equation}
where $r : \cIE M \to \cIE M$ denotes path reversal.

\begin{lemma} If $f \in \cC_{\foe}(\cLE M; \UU(1))$ is simplicially trivial
in the sense that \eqref{E:ffoe_simp_triv} holds for some $e \in \cC(\cIE
M; \UU(1))$ then $f$ is homotopic to $1$ in $\cC_{\foe}(\cL M ; \UU(1)).$
\label{L:strong_ffoe_homotopy} \end{lemma}

\begin{proof} 
The space $\cIE M$ is retractible, by the symmetric retraction to the
midpoint, to $M,$ embedded as the constant paths. Under this retraction,
$e\in\cC(\cIE M;\UU(1))$ is deformed through odd functions to its value,
necessarily $1,$ on constant paths. It can therefore be written uniquely in the
form 
\begin{equation}
\begin{gathered}
e=\exp(2\pi i\kappa),\\
\kappa\in\cC(\cIE M;\bbR),\ \kappa(r\gamma)=-\kappa(\gamma)\ \forall\ \gamma\in\cIE M\\
\Mand \kappa \big(j(\gamma_1,\gamma_2)\big)=\kappa (\gamma_1)+\kappa (\gamma_2)\ \forall\ (\gamma_1,\gamma_2)\in\cIE
M_{\ev(0)=\ev(2\pi)}\cIE M.
\end{gathered}
\label{FuLoSptoSt.436}\end{equation}
The second identity follows from the fact that at constant loops the quotient
$e\big(j(\gamma_1,\gamma_2)\big)/e(\gamma_1)e(\gamma_2)$ is equal to $1$ so has
a unique continuous logarithm vanishing there.

Now the homotopy $e_t=\exp(2\pi it\kappa)$ to $1$ is through odd functions
on $\cIE M$ which satisfy the identities for $e$ in \eqref{E:ffoe_simp_triv} and hence the
homotopy $f_t(\gamma_1,\gamma_2)= e_t^{-1}(\gamma_1)e_t(\gamma_2)$ lies in $\cC_{\foe}\big(\cLE
M;\UU(1)\big)$.
\end{proof}

\begin{corollary} Two elements of $\cC_{\foe}(\cL M; \UU(1))$ are homotopic 
if and only if they are homotopic in $\cC_{\fsn}(\cL M;\UU(1)).$
\label{C:weak_stong_homotopy_equivalence}
\end{corollary}

\begin{proof}
Using multiplicativity, it suffices to consider homotopies of a point
$f\in\cC_{\foe}(\cL M;\UU(1))$ to $1.$ Homotopy in $\cC_{\fsn}(\cL
M;\UU(1))$ implies that $\Rg(f) = 0$ by Proposition~\ref{FuLoSptoSt.410}
and the existence of a homotopy in $\cC_{\foe}(\cL M;\UU(1))$ follows from
Lemma~\ref{L:strong_ffoe_homotopy} and the preceding discussion.
\end{proof}

A two-sided inverse to \eqref{FuLoSptoSt.411} is given by holonomy of
circle bundles over $M$ and this also gives much smoother representatives
of fusive 1-cohomology. We define the corresponding `fusive' functions by
abstraction of the properties of holonomy.

\begin{definition} The space $\CI_{\fus}(\cL M;\UU(1))$ of {\em fusive}
functions consists of those maps \eqref{FuLoSptoSt.155} satisfying
the following three conditions.
\begin{enumerate}
[{\normalfont FF.i)}]
\item\label{FF.L} $f$ is {\em lithe} in the sense of \S\ref{Sect.smooth}, so continuously
differentiable on the energy space and infinitely differentiable on
piecewise smooth and smooth loops with all derivatives in the class of Dirac sections.
\item\label{FF.F} The fusion identity \eqref{FuLoSptoSt.450} holds.
\item \label{FF.R} Under reparametrizations of loops (see \S\ref{Sect.Loops})
\begin{equation}
	\gamma^*f=f^{o(\gamma )}\ \forall\ \gamma\in\Rp(\bbS)
\label{FuLoSptoSt.156}\end{equation}
where $\Rp(\bbS)$ is the reparametrization semigroup and $o = \pm 1$ is orientation.
\end{enumerate}
\end{definition}

\begin{lemma}\label{FuLoSptoSt.456} The group $\Dff(\bbS)$ is dense in
$\Rp(\bbS)$ in the Lipschitz topology and a lithe function on a
$\Dff(\bbS)$-invariant open set is invariant if it is annihilated by a
dense subset of the smooth vector fields on $\bbS$ in the $L^\infty$ topology.
\end{lemma}

\begin{proof} Making a reflection and rotation it suffices to consider the
closure of the subgroup $\Dff^+_{\{0\}}(\bbS)$ of oriented diffeomorphisms
fixing $0.$ These may be identified with their derivatives, which form the
subspace of $\CI(\bbS)$ of positive functions with integral $2\pi.$ The
Lipschitz topology reduces to the $L^\infty$ topology on the derivative so
the closure corresponds to non-negative elements of $L^\infty(\bbS)$ of
integral $2\pi$ and the smooth subset integrates to $\Rp^+_{\{0\}}(\bbS).$

If $v=v(\theta)d/d\theta$ is a smooth vector field on $\bbS$ then the
action of $\exp(itv)$ on lithe functions is  
\begin{equation}
\frac{d}{dt}\big|_{t=0}f(l\circ\exp(itv))=\int_{\bbS}\langle df(l)(s),v\tau_l(s)\rangle ds
\label{FuLoSptoSt.488}\end{equation}
where $\tau_l(s)$ is the tangent vector field to the loop $l.$ The
assumption that $f$ is lithe implies in particular that its derivative,
interpreted in the sense of the $L^2$ pairing, is a piecewise smooth
function on the circle. It follows that if \eqref{FuLoSptoSt.488} vanishes
for a dense subset of $v\in L^{\infty}(\bbS)$ then it vanishes for all
$v\in\CI(\bbS).$ Although diffeomorphisms close to the identity are not in
general given by the exponentiation of a vector field they are given by the
integration of curves in the tangent space and hence the vanishing of the
differential evaluated on the Lie algebra implies the invariance of the
function.
\end{proof}

One particular example of a Lipschitz map of the circle of winding number $1$ is 
\begin{equation*}
T(\theta))=\begin{cases}2\theta&0 \leq \theta \leq \pi \\
0&\pi \leq \theta \leq 2\pi
\end{cases}
\label{FuLoSptoSt.457}\end{equation*}
which has the property that the pull-back may be identified with fusion
with a trivial path 
\begin{equation}
T^*l=\psi(l,l(0))
\label{FuLoSptoSt.458}\end{equation}
where $l(0)$ is the constant path at the initial point and $l$ is interpreted
as a path. Thus any fusive function satisfies the `trivial' fusion condition
that
\begin{equation}
f(\psi(l, l(0)))=f(l)
\label{FuLoSptoSt.455}\end{equation}

\begin{lemma}\label{FuLoSptoSt.459} The group of fusion figure-of-eight maps
contains the fusive maps.
\end{lemma}

\begin{proof} To check that \eqref{FuLoSptoSt.431} holds consider a
figure-of-eight loop $l = J(l_1,l_2)$ constructed in
\eqref{E:figure_of_eight}. This is the rotation by $\pi/2$ of the fusion of the two loops 
\[
\begin{aligned}
	\tilde l_1 &= T^* R(\pi)^* l_1 = \psi(R(\pi)^* l_1, l_1(\pi)), \Mand\\
	\tilde l_2 &= R(\pi)^* T^* l_2 = \psi(l_2(0), l_2). \\
	l &= J(l_1,l_2) = R(\pi/2)^* \psi(R(\pi)^* l_1,l_2).
\end{aligned}
\]
It then follows from the fusion and rotation-invariance of $f$, along with
\eqref{FuLoSptoSt.455} that
\begin{equation*}
f(l) = f(\psi(R(\pi)^*l_1,l_2)) = f(\psi(R(\pi)^* l_1,l_1(\pi)))\,f(\psi(l_2(0),l_2))
= f(l_1)\,f(l_2).
\label{FuLoSptoSt.460}\end{equation*}
\end{proof}

\begin{lemma}\label{FuLoSptoSt.423} For elements of $\CI_{\fus}(\cL
  M;\UU(1)),$ smooth homotopy in $\CI_{\fus}(\cL M;\UU(1))$ is equivalent
  to homotopy in $\cC_{\fsn}(\bbS;\UU(1)).$
\end{lemma}

\begin{proof} It suffices to show that if $f\in\CI_{\fus}(\cL M;\UU(1))$
  and $1$ are homotopic in $\cC_{\fsn}(\cL M;\UU(1)))$ then they are homotopic in
$\CI_{\fus}(\cL M;\UU(1)).$ Thus there is a continous map $F:[0,1]\times\cLE
M\longrightarrow \UU(1)$ with $F(t)\in\cC_{\fsn}(\cL M;\UU(1))$ for all $t,$
$F(0)=$ and $F(1)=f$ and hence a continuous
$G:[0,1]\times\cLE M\longrightarrow \bbR$ with $G(0)=0$ and $F=\exp(2\pi
iG).$ It follows that $G$ satisfies the additive form of the fusion
condition
\begin{equation}
\pi_{12}^*G+\pi_{23}^*G=\pi_{13}^*G\Mon [0,1]\times\cIE^{[3]}M
\label{FuLoSptoSt.424}\end{equation}
since both sides restrict to $1$ at $t=0$ and exponentiate to the same
continuous function. Now, the restriction $g=G(1),$ satisfies $\exp(2\pi
ig)=f$ and hence is locally determined up to a constant. It
follows that it is lithe and reparameterization invariant. So $\exp(2\pi
itg)$ is a fusive homotopy of $f$ to $1.$
\end{proof}

If $L$ is a smooth circle bundle with principal $\UU(1)$-connection bundle
over $M,$ then the holonomy $h(l)$ of a loop $l\in\cLE M$ is defined in
terms of the paths in the total space, $\cIE L,$ which are parallel. For
each point $\theta_0\in\bbS$ and each $e \in L_{l(\theta_0)}$ there is a
unique such path $\tilde l_{e,\theta}$ with initial point $e$ which
projects to $l(\theta + \theta_0).$ Its end-point is $h(l)e$ where $h(l)$
is independent of the choice of $\theta_0$ or $e$ and so defines the
holonomy $h:\cLE M\longrightarrow \UU(1).$

\begin{proposition}\label{FuLoSptoSt.419} The holonomy of a smooth circle
bundle over $M$ with a smooth connection is an element of $\CI_{\fus}(\cL
M;\UU(1))$ and the induced map into $H^1_{\fus}(\cL M)$ descends to a
well-defined group homomorphism 
\begin{equation}
\Tg_{\fus}:H^2(M;\bbZ)\longrightarrow H^1_{\fus}(\cL M).
\label{FuLoSptoSt.422}\end{equation}
\end{proposition}

\begin{proof} The holonomy of two isomorphic bundles with related
connections is the same. The space of connections on a given circle bundle
is affine and homotopic connections give smoothly homotopic holonomies, so
the map \eqref{FuLoSptoSt.422} is well-defined. The holonomy of a tensor
product with tensor product connection is the product of the holonomies so
$\Tg_{\fus}$ is a group homomorphism.
\end{proof}

\begin{theorem}\label{H1fus} The `enhanced trangression map' $\Tg_{\fus}$
and the regression map, $\Rg,$ are inverses of each other and give a
commutative diagram with standard trangression and the forgetful map
\eqref{FuLoSptoSt.212}
\begin{equation}
\xymatrix{
H^2(M;\bbZ)\ar[dr]_{\Tg}&H^1_{\fus}(\cL M)\ar[l]_{\Rg}^{\simeq}\ar[d]\\
&H^1(\cL M;\bbZ).
}
\label{FuLoSptoSt.213}\end{equation}
\end{theorem}

\begin{proof} To prove that $\Rg \circ \Tg_\fus = \Id$, consider a circle
bundle $L$ with connection over $M$ having Chern class $\beta \in H^2(M;
\bbZ),$ along with its holonomy function $h$ representing
$\Tg_\fus(\beta).$ We proceed to show that if $L_h$ is the bundle over $M^2$
constructed from $h$ then its Chern class is $\pi_2^*\beta -\pi_1^*\beta.$

Recall that $L_h$ is constructed as the quotient of the trivial bundle
$\cIE M \times \UU(1) \longrightarrow \cIE M$ by the equivalence relation
determined by $h:\cIE^{[2]}M\longrightarrow \UU(1)$ on the fibers over
$M^2.$ The original bundle $L$ can be pulled back to the path space under
the initial-point map giving the larger fiber space $\cIE M \times_{\ev(0)}
L \longrightarrow M^2$, and then taking the quotient of the trivial bundle
$\cIE M \times_{\ev(0)} L \times \UU(1)$ by the relation determined by the
pull-back of $h$ to $(\cIE M \times_{\ev(0)} L)^{[2]} \equiv \cIE^{[2]} M
\times_{\ev(0)} L^{[2]}$ defines the same bundle $L_h.$

Now consider a second bundle on $M^2$ arising from the quotient of a trivial bundle,
which is related to $L$ itself. The circle 
bundle $L^{-1}\boxtimes L\longrightarrow M^2$ has Chern class $\pi_2^* \beta -
\pi_1^* \beta$ and is defined by the quotient of the trivial bundle $L\boxtimes L\times \UU(1) \longrightarrow L\boxtimes L$ by the relation determined by
\begin{equation}
\begin{gathered}
\delta : (L\boxtimes L)^{[2]} \longrightarrow \UU(1), \\
(e_1\otimes e_2, e'_1\otimes e'_2) = (e_1\otimes e_2, z_1e_1\otimes z_2e_2)
\longmapsto z_1^{-1}z_2 \in \UU(1)
	\label{E:twisted_difference_U(1)}
\end{gathered}
\end{equation}
on the fibers over $M^2.$

We exhibit a map of fiber spaces over $M^2$ with respect to which the two
equivalence relations are related by pull-back. The map
\[
\xymatrix{
\sigma : \cIE M \times_{\ev(0)} L\ar[r]&\cIE L\ar[r]&L\boxtimes L
}	
\]
is defined by extending a pair $(\gamma,l)$ to the covariant constant section
$\wt \gamma\in\cIE L$ over $\gamma$ with initial point $l$ and mapping it to the pair
$(\wt \gamma(0) = l,\wt\gamma(2\pi)).$ The pull-back of
\eqref{E:twisted_difference_U(1)} along the induced map
\[
\sigma^{[2]} : (\cIE M \times_{\ev(0)} L)^{[2]} = \cIE^{[2]} M \times_{\ev(0)} L^{[2]} 
\longrightarrow (L\boxtimes L)^{[2]} 
\]
is well-defined and coincides with the pull-back of the holonomy $h$ of
$L.$ Thus the map $\sigma \times \Id$ of trivial bundles descends to a
bundle isomorphism $L_h \longrightarrow L^{-1}\boxtimes L$ over $M^2,$ and
it follows that $\Rg$ and $\Tg_\fus$ are inverse isomorphisms.

Finally, it remains to show that the diagram \eqref{FuLoSptoSt.213} commutes.
The transgression map is derived from the push-forward map in
cohomology. Namely for any $k,$ using the evaluation map
\begin{equation}
\begin{gathered}
\ev:\bbS\times\cL M\ni(\theta,u)\longmapsto u(\theta)\in M,\\
\xymatrix{
H^k(M;\bbZ)\ar[r]^-{\ev^*}\ar[d]_{\Tg} &H^k(\bbS\times\cL M;\bbZ)\ar@{=}[d]\\
H^{k-1}(\cL M;\bbZ)&\ar[l]H^k(\cL M;\bbZ)+\bbZ\otimes H^{k-1}(\cL M;\bbZ).
}
\end{gathered}
\label{FuLoSptoSt.224}\end{equation}

If a circle bundle $L$ represents a class in $H^2(M;\bbZ)$, then $\ev^*L$
represents the pulled back class in $H^2(\bbS\times \cL M; \bbZ).$ Explicitly,
$\ev^* L$ may be trivialized locally by sections $s_j$ over sets of the form
$\bbS \times \Gamma_j$, where the contractible sets $\Gamma_j =
\Gamma(l_j,\epsilon)$ as in \S\ref{Sect.Loops} form a countable good cover of $\cL
M$, and then the Chern class of $\ev^* L$ is represented by the \v Cech cocycle 
\[
\begin{gathered}
\delta_{ij} : \bbS \times \Gamma_{ij} \longrightarrow \UU(1) \\
s_i = \delta_{ij} s_j, \text{ on } \bbS\times \Gamma_{ij},\
\Gamma_{ij}= \Gamma_i \cap \Gamma_j.
\end{gathered}
\]

The push-forward of this class is represented by the winding number cocycle
$w_{ij} : \Gamma_{ij} \longrightarrow \bbZ$ of
$\delta_{ij}(\theta,l)\delta^{-1}_{ij}(0,l)$ and for later use we prove
this more generally.

\begin{lemma}\label{L:cech_pushforward} 
If $X$ is a paracompact manifold and $\{\Gamma_i\}$ is a good open cover of it then
the $\UU(1)$-\v Cech cohomology of the cover $\{\bbS\times \Gamma_i\}$ of
$\bbS\times X$ is the cohomology $H^*(X; \bbZ)\oplus H^{*-1}(X; \bbZ)$ of $\bbS\times
X$ with the push-forward to $H^{*-1}(X;\bbZ)$ represented by the winding number
of a \v Cech cocycle 
\begin{equation}
\begin{gathered}
\CI\big(\bbS\times\Gamma_*;\UU(1)\big)\ni u_*\longmapsto \alpha_*(2\pi,x)\in\CI(\Gamma_*; \bbZ)\\
u_*(\theta,x)u_*(0,x)^{-1}=\exp\big(2\pi i \alpha_*(\theta,x)\big)
\Mon[0,2\pi]\times \Gamma_*,\ \alpha_*(0,x)=0.
\end{gathered}
\label{17.7.2013.2}
\end{equation}
\end{lemma}

\begin{proof} 
The cohomology of the complex maps to the cohomology of $\bbS\times X$
through any good refinement, such as the products $U_a\times \Gamma _i$ for the
cover of $\bbS$ by the three open intervals 
$(-\epsilon ,2\pi/3+\epsilon ),$ $(2\pi/3-\epsilon ,4\pi/3+\epsilon)$ and
$(4\pi/3-\epsilon ,2\pi+\epsilon ).$ The map to $H^*(X;\bbZ)$ is then given by
restriction to $\{0\}\times \Gamma _*.$ This is surjective since a cocycle can
be lifted to be constant on $\bbS.$ The null space of this restriction map
consists of the cocycles which are equal to $1$ at $\{0\}\times \Gamma _*.$
These have preferred normalized logarithms on the $U_a\times \Gamma _*$
starting at $0$ and the boundary, taken just with respect to the $\bbS$ factor, 
of the resulting real class is the integral
winding-number \v Cech cocycle concentrated in $U_3\times\Gamma _*$ and hence
projecting to $H^{*-1}(X; \bbZ).$ This map is also surjective since an integral
cocycle $k_*$ can be lifted to the $\UU(1)$ cocycle $\exp(2\pi i k_*).$ The
null space of the combined map to $H^*(X;\bbZ)\oplus H^{*-1}(X; \bbZ)$ consists of the cocycles
with global consitent logarithms which are therefore exact.
\end{proof}

Returning to the proof of Theorem~\ref{H1fus}, the trivializing sections $s_j$
may be compared to the parallel lifts $\tilde l$ (as paths) of the loops $l$
with respect to a connection on $L$:
\[
s_j(\theta,l) = h_j(\theta,l) \wt l_{s_j(0,l)}(\theta) \in L_{l(\theta)},\
	h_j : [0,2\pi] \times \Gamma_j \longrightarrow \UU(1).
\]
Here $h_j(2\pi,l)^{-1}$ is the holonomy of $l,$ and since $h_j(0,l) \equiv 1$ there
are normalized logarithms $\eta_j : [0,2\pi] \longrightarrow \bbR$, $\eta_j(0) =
0$ such that $h_j(\theta,l) = \exp\big(2\pi i \eta_j(\theta,l)\big).$ It follows that
the winding number cocycle of 
\[
\delta_{ij}(\theta,l)\delta^{-1}_{ij}(0,l) =
\exp\big(2\pi i (\eta_i(\theta,l) - \eta_j(\theta,l))\big)
\]
is given by $w_{ij}(l) = \eta_i(2\pi,l) - \eta_j(2\pi,l),$
which is precisely the $\check C^1(\cL M; \bbZ)$ lift of the inverse holonomy
$h_\ast(2\pi,\cdot)^{-1} \in \check C^0(\cL M; \UU(1)).$
\end{proof}

Note that there are several accounts in the literature of the recovery of an
abelian principal bundle with connection on $M$ from its holonomy function,
see Teleman, \cite{teleman1963connexions}, Barrett, \cite{barrett1991holonomy}, and
Waldorf \cite{waldorf2009transgressionI}.

\section{Fusive circle bundles}\label{SectFuCiBu}

We proceed to the analog of \eqref{FuLoSptoSt.213} in the next topological
degree. Using \v Cech arguments this can be extended to all degrees, but
here we are interested in geometric realizations of the fusion classes
through circle (or equivalently Hermitian line) bundles. We proceed very
much as in the previous section, first examining the fusion condition alone
in the topological setting, then adding the figure-of-eight condition to
define the notion of fusive 2-cohomology of the loop space in terms of
topological circle bundles satisfying these conditions. The regression map
is defined through bundle gerbes in the sense of Murray and the inverse,
enhanced transgression map is defined via holonomy from principal
$\PU$-bundles, again giving smooth and reparametrization invariant
representatives.

The notion of fusion for a circle bundle is due to Waldorf,
\cite{waldorf2010transgressionII,waldorf2012transgressionIII}. He obtains
an equivalence of categories between bundle gerbes on $M$ and fusion
principal bundles on the loop space which are equivariant with respect to
thin homotopies via a general transgression functor. We consider bundles
with the formal properties of the holonomy bundles of $\PU$ bundles to find 
appropriately regular representatives.

\begin{definition}\label{FuLoSptoSt.425} A fusion circle bundle over $\cL M$ is a
(locally trivial) topological circle bundle $D\longrightarrow \cLE M$
with a continuous fusion isomorphism 
\begin{equation}
\Phi:\psi_{12}^*D\otimes\psi_{23}^*D\longrightarrow \pi_{13}^*D\Mover\cIE^{[3]}M
\label{FuLoSptoSt.426}\end{equation}
such that the diagram of circle bundle isomorphisms over $\cIE^{[4]}M$
\begin{equation}
\xymatrix{
\psi_{12}^*D\otimes\psi_{23}^*D\otimes\psi_{34}^*D\ar[r]^-{\Phi\otimes\Id}\ar[d]_{\Id\otimes\Phi}&
\psi_{13}^*D\otimes\psi_{34}^*D\ar[d]^{\Phi}\\
\psi_{12}^*D\otimes\psi_{24}^*D\ar[r]^-{\Phi}&
\psi_{14}^*D
}
\label{FuLoSptoSt.427}\end{equation}
commutes. A continuous isomorphism between such bundles is fusion if it
lifts to intertwine the fusion maps.
\end{definition}

An automorphism of a circle bundle is a map into $\UU(1)$ and such an
automorphism of a fusion circle bundle is fusion if and only if this
function is a fusion function. Thus the homotopy classes of fusion
automorphisms of any fusion bundle may be identified with the group
\eqref{FuLoSptoSt.451}.

If $J$ is a circle bundle over the path space $\cIE M$ then
$D=\pi_1^*J\otimes\pi_2^*J^{-1}$ over $\cIE^{[2]}M=\cLE M$ has the `product'
fusion structure given by the pairing 
\begin{equation*}
\pi_{12}^*D\otimes\pi_{23}^*D=
\pi_1^*J\otimes\pi_2^*J^{-1}\otimes\pi_2^*J\otimes\pi_3^*J^{-1}\longrightarrow
\pi_{13}^*D\Mover\cIE^{[3]}M.
\label{FuLoSptoSt.461}\end{equation*}

\begin{proposition}\label{FuLoSptoSt.429} Each fusion circle bundle over
$\cLE M$ defines a (topological) bundle gerbe over $M^2$ and is `fusion
trivial' i.e.\ is isomorphic to a trivial bundle with trivial fusion
isomorphism, if and only if the bundle gerbe is trivial.
\end{proposition}

\begin{proof}  Given a fusion circle bundle $D,$ consider the diagram
formed from the free path space fibration and fusion map:
\begin{equation}
\xymatrix{
&\psi^*D\ar[d]&D\ar[d]\\
\cIE M\ar[dr]&\cIE^{[2]}M\ar[d]\ar[r]^{\psi}\ar@<.5ex>[l]^-{\pi_1}
\ar@<-.5ex>[l]_-{\pi_2}&\cLE M\\
&M^2.
}
\label{FuLoSptoSt.11}\end{equation}
The fusion conditions \eqref{FuLoSptoSt.426} and \eqref{FuLoSptoSt.427} are
equivalent to the condition that $\psi^* D$ defines a bundle gerbe, and
this gerbe is (simplicially) trivial if and only if there is a
circle bundle $J$ over $\cIE M$ and an isomorphism
\begin{equation}
D\longrightarrow \pi_1^*J\otimes\pi_2^*J^{-1}
\label{FuLoSptoSt.462}\end{equation}
which induces the fusion isomorphism $\Phi.$ Since $\cIE M$ retracts onto the
constant loops $M$ at the initial point, $J$ is isomorphic to the pull-back
under $\pi_1$ of a bundle $\tilde J$ on $M.$ Tensoring $J$ with
the pull-back of $\tilde J^{-1}$ does not change the isomorphism
\eqref{FuLoSptoSt.462} since it cancels in the tensor product. Thus we may
assume that $J$ is trivial, and then \eqref{FuLoSptoSt.462} is a fusion
isomorphism to a trivial bundle.

Conversely, if $D$ is fusion trivial then $\psi^* D$ is trivial and this implies
triviality of the bundle gerbe.
\end{proof}

Murray has shown that bundle gerbes up to simplicial triviality are
classified by the Dixmier-Douady class. For later reference, we briefly
recall how the Dixmier-Douady class is defined, and how the simplicial
trivialization is constructed, in the case of the bundle gerbe above.

A cover $\set{B_i}$ of $M$ by small geodesic balls leads to a good open cover
$\set{B_i\times B_j}$ of $M^2$. Smooth local sections 
\begin{equation}
	s_{i,j} : B_i \times B_j\longrightarrow \cI M \subset \cIE M 
	\label{E:path_sections}
\end{equation}
can be chosen and then 
\begin{equation}
D_{ik,jl} = (s_{i,j},s_{k,l})^* \psi^*D  \longrightarrow
B_{ik}\times B_{jl} = (B_i \cap B_j) \times (B_j \cap B_l)
\label{E:transition_bundle}
\end{equation}
are circle bundles over the double intersections. Using the fusion
isomorphism, sections $\lambda_{ik,jl} : B_{ik}\times 
B_{jl} \longrightarrow D_{ik,jl}$ of these may be compared over triple
intersections:
\begin{equation}
\begin{gathered}
\Phi(\lambda_{ik,jl}, \lambda_{km,ln}) = \sigma_{ikm,jln}\,\lambda_{im,jn},
\Mon B_{ikm} \times B_{jln} \\
\sigma_{ikm,jln} : B_{ikm} \times B_{jln} \longrightarrow \UU(1)
\end{gathered}
\label{E:DD_Cech_FCB}\end{equation}
resulting in a \v Cech cocycle $\sigma.$ This represents the Dixmier-Douady class of $D$ in
$H^3(M^2;\bbZ)$ where the integral 3-cocycle is formed by the \v Cech boundary of
logarithms of the $\sigma _{ikm,jln}.$

Since the $B_i$ form a good open cover, triviality of this class means that
$\sigma$ is a boundary and hence that the sections $\lambda_{ik,jl}$ can be
chosen in such a way that \eqref{E:DD_Cech_FCB} holds with $\sigma \equiv
1.$ Then the bundle $J$ may be defined by gluing the bundles
\begin{equation}
J_{i,j} \longrightarrow \cIE M \big|_{B_i \times B_j},\ (J_{i,j})_{\gamma}
= D_{\psi(\gamma,s_{i,j}(\ev(\gamma)))} 
\label{E:local_J}\end{equation}
using the compatible gluing isomorphisms formed by the $\lambda_{ik,jl}.$
This gives the global bundle $J$ over $\cIE M$ which trivializes $D$ simplicially.

\begin{proposition}\label{FuLoSptoSt.464} The Dixmier-Douady class of the
bundle gerbe \eqref{FuLoSptoSt.11} defined by a fusion circle bundle
induces an isomorphism
\begin{multline}
\dD:\{\text{Fusion circle bundles}\}/\text{Fusion isomorphism}\\
\longrightarrow
\{\delta \in H^3(M^2;\bbZ);\delta \big|_{\Diag}=0\}.
\label{FuLoSptoSt.463}\end{multline}
\end{proposition}

\begin{proof} The Dixmier-Douady class, characterizing the bundle gerbe up
to simplicial triviality, lies in $H^3(M^2;\bbZ)$ and behaves naturally
under restriction. Over the diagonal, the bundle gerbe has a trivial
subgerbe; indeed, the inclusion map defines a bundle gerbe morphism from
$D$ restricted to constant paths $M \subset \cIE M$ as a trivial fibration
over $M=\Diag.$ Thus the map has range in the space indicated in
\eqref{FuLoSptoSt.463}. The injectivity of this map has already been
established by Proposition~\ref{FuLoSptoSt.429} since the vanishing of the
Dixmier-Douady class implies the simplicial triviality of the bundle gerbe
and hence the existence of a fusion isomorphism to a trivial bundle.

The surjectivity of $\dD$ is established below in the discussion of the
holonomy of principal $\PU$ bundles. More directly, it suffices to note
that each class in $H^3(M^2;\bbZ)$ is represented by a (topological)
principal $\PU$ bundle over $M^2$ and if the class lies in the range space
in \eqref{FuLoSptoSt.463} then the bundle has a section over the diagonal
and the pull-back over $\cIE M$ has a global section since the latter
retracts to $M \equiv \Diag$ --- this is constructed explicitly in the
smooth case by parallel transport below. The two values of the section at a
point of $\cIE^{[2]}M$ are related by an element of $\PU$ mapping the
second to the first since they are in the same fiber of the original bundle,
and this defines a continuous (essentially holonomy) map
\begin{equation}
h:\cLE M\longrightarrow \PU
\label{FuLoSptoSt.465}\end{equation}
which is multiplicative under fusion 
\begin{equation}
\pi_{12}^*h\cdot\pi_{23}^*h=\pi_{13}^*h\Mon\cIE^{[3]}M.
\label{FuLoSptoSt.466}\end{equation}
The pull-back of the canonical bundle $\UU/\PU$ by $h$ is a circle bundle
over $\cLE M$ with fusion isomorphism given by the product identification
of the canonical bundle, i.e.\ the product in $\UU$ and compatibility
condition over $\cIE^{[4]}M.$ Moreover, the bundle gerbe defined by this
bundle is a subgerbe of the inverse of the gerbe with total space the
pull-back of the original $\PU$ bundle to $\cIE M.$ Thus it represents the
inverse of the original class and it follows that \eqref{FuLoSptoSt.463} is
surjective and hence is an isomorphism.
\end{proof}

Proceeding as in \S\ref{SectFuMa} to refine \eqref{FuLoSptoSt.463} to a map
into the 3-cohomology of $M$ we add the analog of the condition in
Definition~\ref{FuLoSptoSt.407} involving the figure-of-eight product
\eqref{E:figure_of_eight}.

\begin{definition}\label{FuLoSptoSt.467} A {\em fusion-figure-of-eight structure}
on a circle bundle $D$ with fusion isomorphism \eqref{FuLoSptoSt.426},
\eqref{FuLoSptoSt.427} is an isomorphism 
\begin{equation}
\eta:J^*D\longrightarrow \Pi_1^*D\otimes \Pi_2^*D\Mover \cLE M\times_{\ev(0) = \ev(\pi)} \cLE M
\label{FuLoSptoSt.468}\end{equation}
which is compatible with the fusion isomorphism in the sense that
\begin{equation}
\begin{gathered}
\bpns{(\gamma_1,\gamma_2,\gamma_3),(\gamma'_1,\gamma'_2,\gamma'_3)} 
\in \pi_{12}^*\cIE^{[3]}\times_{M^3}\pi_{23}^*\cIE^{[3]}
M\Mst\gamma_i(2\pi) = \gamma'_i(0)\Longrightarrow \\
\xymatrix{
D_{\psi(j(\gamma_1,\gamma'_1),j(\gamma_2,\gamma'_2))}\otimes
D_{\psi(j(\gamma_2,\gamma'_2),j(\gamma_3,\gamma'_3))}
\ar[d]_-{\eta\otimes\eta }
\ar[r]^-{\Phi}
&
D_{\psi(j(\gamma_1,\gamma'_1),j(\gamma_3,\gamma'_3))}
\ar[d]^-{\eta}
\\
D_{\psi(\gamma_1,\gamma_2)}\otimes D_{\psi(\gamma'_1,\gamma'_2)}\otimes
D_{\psi(\gamma_2,\gamma_3)}\otimes D_{\psi(\gamma'_2,\gamma'_3)}
\ar[r]_-{\Phi\otimes\Phi}
&
D_{\psi(\gamma_1,\gamma_3)}\otimes D_{\psi(\gamma'_1,\gamma'_3)}
}
\end{gathered}
\label{FuLoSptoSt.472}\end{equation}
commutes (see Figure~\ref{F:fusion_fo8}); the bundle is then said to be a fusion-figure-of-eight circle bundle.
\end{definition}

\begin{figure}[t]
\begin{tikzpicture}[rotate=-25]
\draw (-2,-2) .. controls (-1,-2) and (0,-1) .. node[below] {$\gamma_1$} (0,0) .. controls (1,0) and (2,1) .. node[below] {$\gamma'_1$} (2,2);
\draw (-2,-2) .. controls (-2,-1) and (-1,0).. node[above] {$\gamma_3$} (0,0) .. controls (0,1) and (1,2) .. node[above] {$\gamma'_3$} (2,2);
\draw (-2,-2) -- node[below] {$\gamma_2$} (0,0) -- node[below] {$\gamma'_2$} (2,2);
\end{tikzpicture}
\caption{Compatibility of fusion and figure-of-eight.}
\label{F:fusion_fo8}
\end{figure}
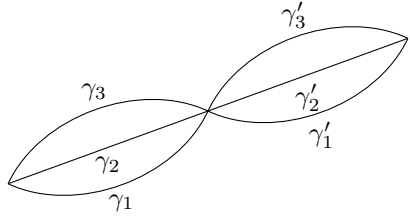
Again by analogy with the standard case we define fusion 2-cohomology as 
\begin{equation}
H^2_{\fus}(\cL M)=\{\text{Fusion-figure-of-eight circle bundles}\}\big/
\text{Fusion isomorphisms.}
\label{FuLoSptoSt.428}\end{equation}
In contrast to the case of functions, it is not clear that the equivalence
relation here can be strengthened to isomorphisms intertwining both the fusion
and figure-of-eight structures.

\begin{proposition}\label{FuLoSptoSt.473} The restriction to
fusion-figure-of-eight circle bundles of the map \eqref{FuLoSptoSt.463} takes
values in $\{\pi_2^*\beta-\pi_1^*\beta \in H^3(M^2;\bbZ);\beta \in
H^3(M;\bbZ)\}$ and the resulting regression map 
\begin{equation}
\Rg:H^2_{\fus}(\cL M)\longrightarrow H^3(M;\bbZ),\ \Rg(D)=\beta 
\label{FuLoSptoSt.474}\end{equation}
is injective.
\end{proposition}
\noindent It will be proved below that $\Rg$ is an isomorphism.
\begin{proof} It is only necessary to see that $\Rg$ is well-defined, its
injectivity then follows from Proposition~\ref{FuLoSptoSt.464}.

Let $D \longrightarrow \cLE M$ be a fusion-figure-of-eight circle bundle,
and $\cG = (D, \cIE M, M^2)$ its associated bundle gerbe.  The
figure-of-eight condition implies that there is a gerbe morphism
$\pi_{12}^* \cG \otimes \pi_{23}^* \cG \longrightarrow \pi_{13}^* \cG$ over
$M^3$ where the map on fiber spaces is the join \eqref{E:join}; thus the
Dixmier-Douady class $\alpha \in H^3(M^2; \bbZ)$ of $\cG$ satisfies
\[
\pi_{23}^* \alpha - \pi_{13}^* \alpha + \pi_{12}^* \alpha = 0 \in H^3(M^3;\bbZ).
\]
Pulling this equation back to $H^3(M^2; \bbZ)$ by the embedding $i_2 : M^2
\hookrightarrow \bar m \times M^2 \subset M^3$ as in the proof of
Proposition~\ref{FuLoSptoSt.410} shows that $\alpha = \pi_2^* \beta -
\pi_1^* \beta,$ where $\beta = i_1^* \alpha \in H^3(M; \bbZ)$ is the
pull-back of $\alpha$ to $H^3(M; \bbZ)$ along the embedding $i_1 : M
\hookrightarrow \bar m \times M \subset M^2.$
\end{proof}

To complete the proof that the regression map is an isomorphism and find
smooth representatives of these fusion classes we consider much more
restrictive properties obtained by abstraction from those of the holonomy
bundles of principal $\PU$ bundles, as shown subsequently in
Proposition~\ref{FuLoSptoSt.380}.

\begin{definition}\label{FuLoSptoSt.244} A circle bundle $D$ over $\cL M$
is \emph{fusive} if it satisfies the following four conditions.

\begin{enumerate}
[{\normalfont FB.i)}]
\item \label{FBlithe} $D$ is \emph{lithe} in the sense that it has trivializations
as a principal $\UU(1)$ bundle over the open sets $\Gamma(l,\epsilon)$ as
in \eqref{FuLoSptoSt.37} for some $\epsilon >0$
\begin{equation}
T_{l}:D\longrightarrow \Rp\cdot\Gamma_E(l,\epsilon )\times\UU(1)
\label{FuLoSptoSt.225}\end{equation}
and the transition maps 
\begin{equation}
T_{ll'}=T_l\cdot T_{l'}^{-1}:\Rp\cdot\Gamma_E (l,\epsilon )\cap
\Rp\cdot\Gamma_E (l',\epsilon )\longrightarrow \UU(1)
\label{FuLoSptoSt.226}\end{equation}
are lithe in the sense of \S\ref{Sect.smooth}.
\item \label{FBfus} There is a fusion isomorphism for $D,$ a lithe bundle isomorphism 
\begin{equation}
\Phi:\fusm_{12}^*D\otimes\fusm_{23}^*D\longrightarrow \fusm_{13}^*D\Mover
\cI^{[3]}M,\ \fusm_{ij}=\pi_{ij}\circ \fusm
\label{FuLoSptoSt.167}\end{equation}
which is associative in the sense that the two iterated
maps over $\cI^{[4]}M$ are equal:
\begin{equation*}
\Phi(\Phi,\cdot)=\Phi(\cdot,\Phi).
\label{FuLoSptoSt.172}\end{equation*}
\item \label{FBreparam} $D$ is strongly parameter-independent, in the sense that there is a
lithe $\UU(1)$ bundle isomorphism
\begin{equation}
A:R^*D\longrightarrow \pi_2^*D^{\pm 1}\Mover \Rp^\pm(\bbS)\times\cL M,\ R(f,l)=l\circ f
\label{FuLoSptoSt.166}\end{equation}
and this isomorphism is consistent with some choice of trivializations
\eqref{FuLoSptoSt.226} in the sense that $A$ becomes the identity
transformation on the trivial bundle, near the identity in $\Dff^+(\bbS).$

\item \label{FBreparamfus} A consistency condition holds between the fusion
and the reparameterization isomorphisms. Namely if $r_i\in\Rp^+([0,2\pi])$
are, for $i=1,$ $2,$ $3,$ reparametrizations of the interval 
and $\chi_{ij}$ are their fusions to elements of
$\Rp_{\set{0,\pi}}^+(\bbS)$ (fixing the points $0$ and $\pi$) then for the fusion of three
paths $(\gamma_1,\gamma_2,\gamma_3)\in\cI^{[3]}M$ to loops $l_{ij},$ the
diagram
\begin{equation}
\xymatrix{D_{l_{12}\circ \chi _{12}}\otimes
D_{l_{23}\circ\chi_{23}}\ar[r]^-{\Phi}\ar[d]_{A\otimes A}&
D_{l_{13}\circ \chi_{13}}\ar[d]^A\\
D_{l_{12}}\otimes D_{l_{23}}\ar[r]^-{\Phi}&D_{l_{13}}
}
\label{FuLoSptoSt.230}\end{equation}
commutes.
\end{enumerate}
\end{definition}

In brief then, a fusive circle bundle over $\cL M$ is a lithe circle bundle
with locally trivial lithe reparameterization and fusion isomorphisms and
with the consistency condition holding between reparameterization and fusion.

The tensor product of two fusive circle bundles is again fusive, as is the
inverse of a fusive bundle.

We distinguish between two notions of isomorphism between such circle
bundles, namely {\em lithe-fusion} and {\em fusive} isomorphisms. For the
former, which we use initially, there is merely required to be a lithe
isomorphism between the two bundles which intertwines the fusion
isomorphisms \eqref{FuLoSptoSt.167}. For a fusive isomorphism we require in
addition that the isomorphism intertwines the reparameterization
isomorphisms. It is shown below that these induce the same equivalence
relation on fusive circle bundles.

We believe, but do not show here, that the collection of lithe circle
bundles -- just satisfying FB.\ref{FBlithe}) above -- modulo lithe
isomorphisms, may be identified with $H^2(\cL M;\bbZ)$ just as in the
finite-dimensional case.

Recall that the figure-of-eight product on loops can be written in terms of
fusion and reparametrization. The same considerations as in the proof of
Lemma~\ref{FuLoSptoSt.459} show

\begin{lemma}\label{FuLoSptoSt.476} Fusive circle bundles are
fusion-figure-of-eight and hence there is a natural map
\begin{equation}
\text{\{Fusive circle bundles\}}\big/\text{Lithe-fusion
isomorphisms}\longrightarrow H^2_{\fus}(\cL M).
\label{FuLoSptoSt.477}\end{equation}
\end{lemma}
\noindent It is shown below that \eqref{FuLoSptoSt.477} is an isomorphism.

\begin{proposition}\label{FuLoSptoSt.380} The holonomy of a smooth
principal $\PU$ bundle with connection over $M$ defines a fusive circle
bundle $D$ over $\cL M$ through the pull-back of the canonical circle bundle
over $\PU.$ This circle bundle is independent of choices up
to fusive isomorphism, and induces the enhanced transgression map 
\begin{equation}
\Tg_{\fus}:H^3(M;\bbZ)\longrightarrow H^2_{\fus}(\cL M).
\label{FuLoSptoSt.381}\end{equation}
\end{proposition}

\begin{proof} Let $P$ be a smooth principal $\PU$ bundle over $M.$
Smoothness here means that the bundle has local trivializations with
transition maps smooth as maps into the Banach manifold $\PU.$ A connection
on $P$ can then be constructed by summing the local Maurer-Cartan forms
over a partition of unity relative to an open cover of $M$ by such
trivializations.

Any smooth path $\gamma$ in $\cI M$ can be lifted by parallel transport to a unique
smooth path $\wt \gamma_p \in \cI P$ covering $\gamma$ and having a given initial point
$p \in P_{\gamma(0)}.$ For a loop $l \in \cL M$ the difference at the endpoints
of the lift $\wt l_p \in \cI P$ defines the holonomy
\begin{equation}
\begin{gathered}
	H : P \times_{\ev(0)} \cLE M \longrightarrow \PU \\
	\wt l_p(2\pi) = H(p,l)p \in P_{l(0)} = P_{l(2\pi)}
\end{gathered}
	\label{FuLoSptoSt.170}
\end{equation}
by continuous extension to the energy space.

In contrast to the case of abelian structure groups, the holonomy of a
$\PU$ bundle is not independent of the initial point $p,$ however
\eqref{FuLoSptoSt.170} is equivariant relative to the principal action on
$P$ and the adjoint action on $\PU$ from the fiber equivariance of the
connection on a principal bundle. The canonical circle bundle of the
unitary group over the projective unitary group for a separable
infinite-dimensional Hilbert space, $\UU\longrightarrow \PU$, is also
equivariant for the conjugation action on $\PU$ since this action extends
uniquely to the conjugation action on $\UU.$ Thus, pulling back $\UU$ under
\eqref{FuLoSptoSt.170} gives an equivariant circle bundle and quotient
\begin{equation}
\begin{gathered}
\tilde D =H^*\UU\longrightarrow P \times_{\ev(0)} \cLE M,
\\ D= \tilde D/\PU \longrightarrow \cLE M
\end{gathered}
\label{FuLoSptoSt.344}\end{equation}
by the free $\PU$ action on $P\times_{\ev(0)}\cLE M.$ We call $D$ (or its
restriction to $\cL M$) the \emph{holonomy bundle} associated to $P$ and
ultimately to its Dixmier-Douady class. We proceed to show that $D$ has the
properties required of a fusive circle bundle over $\cL M$ and then discuss the
dependence on choices.

Lithe regularity of $D$ over $\cL M$ is discussed in \S\ref{Sect.smooth}.

Since the reparameterization of a covariant-constant path in $P$ is
covariant-constant, the restricted reparameterization
semigroup acts trivially on $H$ and this defines an action on $D$:
\begin{equation}
\begin{gathered}
H(l\circ r,p)=H(l,p),\ r\in \Rp^+_{\{0\}}(\bbS)\\
A:R^*D\longrightarrow D,\ R\in\Rp^+_{\{0\}}(\bbS).
\end{gathered}
\label{FuLoSptoSt.346}\end{equation}
Here $\Rp_{\set 0}^+(\bbS)$ denotes the reparametrizations which
fix the initial point $0 \in \bbS.$
For rotations, observe that parallel translation along the path defines an
action of the universal cover, $\bbR$, on $P \times_{\ev(0)} \cLE M$ which
covers the $\UU(1)$ rotation action on $\cLE M.$ The pull-back of $H$ by this
action is invariant up to conjugacy in $\PU$ so \eqref{FuLoSptoSt.346} may be
extended to an action of the full reparametrization semigroup.

Now consider an open set $\Gamma_E(l,\epsilon)$ consisting of the loops
which are $\epsilon$-close in the sense of energy to the loop $l.$ The
initial point map has image in a small geodesic ball, $\ev(0):\Gamma
_E(\epsilon ,l)\longrightarrow B\subset M.$ If $p:B\longrightarrow P$ is a
section of $P$ over it, then the image of $H$ restricted to $\Gamma
_E(\epsilon ,l)$ computed at $p\circ\ev(0)$ lies in a ball in $\PU$ which
is small in norm with $\epsilon.$ Thus, if $\epsilon$ is small enough,
$\UU/\PU$ is trivial over the range and hence so is $D.$ The
reparameterization isomorphism is the identity in this trivialization. 

The properties of parallel transport show that if three loops, $l_{12},$
$l_{23}$ and $l_{13}$ are related by fusion then their holonomies are
related by multiplication:
\begin{equation}
H(l_{13},p)=H(l_{23},p)H(l_{12},p),\ p\in P_m
\label{FuLoSptoSt.171}\end{equation}
with $m = l_{ij}(0)$ the common initial point. The classifying bundle $\UU$
pulls back to be equivariantly isomorphic to the product of the classifying
bundles under multiplication $\PU \times \PU \longrightarrow \PU$, and it
follows that $D$ has a fusion isomorphism, the associativity of which
follows from associativity of multiplication in $\PU.$ Litheness follows
from the construction as does the compatibility with restricted
reparameterization.

Finally consider the effect on $D$ of altering the choices made. Two $\PU$
bundles with the same Dixmier-Douady class are smoothly isomorphic,
replacing $P$ by another but transfering the chosen connection leaves the
holonomy, and hence $D$ unchanged.  Changing the connection on $P$ adds
the pull-back from $M$ of a smooth 1-form $\beta
:M\longrightarrow\mathfrak{g}(\PU)\otimes T^*M$ with values in the Lie
algebra of $\PU;$ these form an affine space. Such a 1-form can be lifted
locally, and hence globally, to a 1-form $\tilde \beta$ with values in the
Lie algebra of $\UU.$ The holonomy around a curve is shifted by the
multiplication by the solution of the exponential equation for $\beta$
along the curve and hence is homotopic to the identity. Integrating
$\tilde\beta$ instead lifts this factor into $\UU$ and gives an isomorphism
of the pulled back circle bundles. This isomorphism intertwines the
structure maps for parameter independence and fusion.
\end{proof}

When we consider fusive isomorphisms below we will make use of the following result.
\begin{proposition}\label{P:FCB_connection}
The holonomy bundle $D \longrightarrow \cL M$ of a $\PU$-bundle on $M$ may be
equipped with a lithe connection which is equivariant with respect to fusion
and the action of the restricted reparametrization semigroup $\Rp_{\set{0,\pi}}(\bbS).$
\end{proposition}

\begin{proof} Since $\PU$ is a smooth paracompact Banach manifold the
locally trivial circle bundle $\UU$ can be given a smooth connection
although this will not be invariant under the conjugation action of $\PU.$
Nevertheless, over each element of the cover of $M$ by small geodesic
balls, $B_i,$ there is a smooth trivializing section
$b_i:B_i\longrightarrow P,$ which gives a local slice for the $\PU$ action 
on $P \times_{\ev(0)} \cL M.$ Thus over
\begin{equation}
\Omega_i=\{l\in\cL M;l(0)\in B_i\},\ D\equiv H_i^*(\UU/\PU),\ H_i=H\circ b_i
\label{FuLoSptoSt.345}\end{equation}
where $b_i$ is lifted to a section of the pull-back bundle. In particular, the
chosen connection on $\UU/\PU$ pulls back to give $D$ a lithe connection, $\nabla_i,$
over each $\Omega _i.$ Since the map $H_i$ is restricted-reparameterization-invariant,
meaning with respect to $\Rp_{\set{0}}^+(\bbS),$ so are these
connections. Over the intersection of two such open sets
\begin{equation}
\Omega _{jk}=\Omega _j\cap \Omega _k,\ \nabla_j-\nabla_k= u_{jk}
\label{FuLoSptoSt.347}\end{equation}
the connections differ by a smooth, restricted-reparameterization-invariant,
1-form. Over triple intersections these 1-forms satisfy the cocycle
condition. A partition of unity $\rho _i$ on $M$ subordinate to the open cover
by the $B_i$ pulls back under $\ev(0)$ to a partition
of unity subordinate to the open cover $\Omega_i$ of $\cL M$ and the
1-forms 
\begin{equation}
v_i=\sum\limits_{j\not=i}\rho _ju_{ij}\Mon \Omega _i
\label{FuLoSptoSt.348}\end{equation}
are lithe and restricted-reparameterization-invariant. The shifted local
connections $\nabla_i-v_i$ then patch to a global,
restricted-reparameterization-invariant, connection on $D$ over $\cL M.$

This connection need not respect the fusion structure, but we proceed to show
that it can be modified to have this property while retaining invariance under
the semigroup $\Rp_{\set{0,\pi}}(\bbS)$ of restricted reparametrizations fixing
$0$ and $\pi$,

To do so, consider the bundles
\[
J_{i,j} \longrightarrow G_{ij} = \set{\gamma \in \cI M; \ev(\gamma) \in B_i \times B_j}
\]
defined by \eqref{E:local_J}. The restricted-reparameterization-invariant
connection on $D$ induces a connection $\nabla_{i,j}$ on $J_{i,j}.$ Recall that
$D_{ik,jl}$ defined in \eqref{E:transition_bundle} functions as a transition
bundle, pulled back to $G_{ij} \cap G_{kl}$, with fusion defining an isomorphism 
\begin{equation}
\beta_{ij,kl}:J_{i,j} \cong J_{k,l}\otimes D_{ik,jl}\Mover G_{ij}\cap G_{kl}.
\label{FuLoSptoSt.359}\end{equation}
The transition bundles themselves have a cocycle isomorphism 
\begin{equation}
D_{ik,jl}\otimes D_{km,ln}\otimes D_{mi,nj}\simeq\UU(1), \Mover B_{ikm} \times B_{jln}
\label{FuLoSptoSt.360}\end{equation}
with the consistency condition of a Brylinski-Hitchin gerbe over four-fold
intersections of the products $B_i\times B_j$ in $M^2.$ It follows that one
can choose connections $\nabla_{ik,jl}$ on the $D_{ik,jl}$ such that on
triple intersections of the sets in $M^2$ the tensor product connection is
the trivial one, and likewise such that the tensor product is trivial
with respect to the isomorphisms $D_{ik,jl} \otimes D_{ki,lj} \simeq \UU(1)$ 
on $B_{ij} \times B_{kl} \simeq B_{kl}\times B_{ij}.$

The connections on the $J_{i,j}$ on overlaps $G_{ij}\cap G_{kl}$ can be compared under the
fusion isomorphisms \eqref{FuLoSptoSt.359} and it follows that
\begin{equation}
\nabla_{ij}-\beta _{ij,kl}^*(\nabla_{ik,jl}\otimes \nabla_{kl})=\sigma_{ik,jl}
\label{FuLoSptoSt.361}\end{equation}
is a restricted-reparameterization-invariant 1-form on $G_{ij}\cap G_{kl}.$
Moreover, these 1-forms constitute a \v Cech cocycle on triple intersections
and which, using the pull-back of a partition of unity on the $B_i\times B_j$
in $M^2,$ is the boundary of a \v Cech class of
restricted-reparameterization-invariant 1-forms on the $G_{ij}.$ Thus in fact
the connections on the $J_{i,j}$ may be modified, remaining
restricted-reparameterization-invariant (with respect to
$\Rp_{\set{0,\pi}}(\bbS)$) so that they are identified with the tensor product
connections under \eqref{FuLoSptoSt.359}.  Recalling that
$D=\pi_1^*J_{ij}\otimes \pi_2^*J_{ij}^{-1} = J_{ij}\boxtimes J_{ij}^{-1}$ over
loops which have initial point in $B_i$ and lie in $B_j$ at parameter value
$\pi,$ this induces a global restricted reparameterization-invariant and fusion
connection on $D,$ which is well-defined since the connections induced from
$J_{ij}\boxtimes J_{ij}^{-1}$ and $J_{kl}\boxtimes J_{kl}^{-1}$ agree on
overlaps.
\end{proof}

\begin{theorem}\label{FuLoSptoSt.168} The map \eqref{FuLoSptoSt.381} is an
isomorphism and is the inverse to the regression map \eqref{FuLoSptoSt.474}
giving a commutative diagram
\begin{equation}
\xymatrix{
H^3(M;\bbZ)\ar[dr]_{\Tg}&H^2_{\fus}(\cL M)\ar[l]_{\Rg}^{\simeq}\ar[d]\\
&H^1(\cL M;\bbZ).
}
\label{FuLoSptoSt.173}\end{equation}
\end{theorem}

\begin{proof} Since the regression map \eqref{FuLoSptoSt.474} has already
been shown to be injective, it suffices to show that the fusion
transgression map \eqref{FuLoSptoSt.381} satisfies
\begin{equation}
\Rg\circ\Tg_{\fus}=\Id
\label{FuLoSptoSt.382}\end{equation}
to establish that these are isomorphisms inverse to each other.

Proceeding as in the proof of Theorem~\ref{H1fus}, consider a $\PU$ bundle $P$
over $M$ with Dixmier-Douady class $\alpha \in\ H^3(M;\bbZ),$ and its holonomy
bundle $D$. We exhibit an isomorphism of bundle gerbes over $M^2$; the first
is the pull-back of the gerbe $D \longrightarrow \cI^{[2]} M$ to the fiber space $(\cI
M \times_{\ev(0)} P)^{[2]}$, and the second is a variation on the lifting bundle gerbe of
$P.$ More precisely, consider the fibration ({\em not} a $\PU$-bundle)
\[
	Q= \pi_1^* P \times \pi_2^* P \longrightarrow M^2
\]
with difference map
\[
\begin{gathered}
	\delta : Q^{[2]} \to \PU \\
\big((p_1,p_2),(q_1,q_2)\big) = \big((p_1,p_2),(a_1p_1,a_2p_2)\big)
\longmapsto a_1a_2^{-1}, \ a_1,a_2 \in \PU.
\end{gathered}
\]
The pull-back $\delta^*(\UU/\PU) \longrightarrow Q^{[2]}$ defines a $\UU(1)$
bundle gerbe with Dixmier-Douady class $\pi_1^* \alpha - \pi_2^* \alpha \in
H^3(M^2; \bbZ).$

As in the proof of Theorem~\ref{H1fus}, this gerbe may be pulled back along the map
\[
\sigma^{[2]} : \cI^{[2]} M \times_{\ev(0)} P^{[2]} \equiv (\cI M \times_{\ev(0)} P)^{[2]}
\longrightarrow \cI^{[2]} P \longrightarrow Q^{[2]}
\]
of fiber bundles over $M^2$ induced by the composition $\sigma =
\big(\ev(0)\times \ev(2\pi)\big)\circ \mathrm{par}$ of parallel translation and
evaluation. In fact the pull-back 
\[
(\sigma^{[2]})^*\delta : \cI^{[2]} M \times_{\ev(0)} P^{[2]} \longrightarrow \PU
\]
coincides with the holonomy $H$ evaluated on the first factor of $P$, which leads
to a bundle gerbe isomorphism between $(\pi_1^*D, \cI M\times_{\ev(0)}P, M^2)$ 
and $(\delta^* \UU, \pi_1^* P\times \pi_2^* P, M^2)$, so these have the same Dixmier-Douady
class $\pi_1^* \alpha - \pi_2^* \alpha$ and \eqref{FuLoSptoSt.382} follows.

To prove the commutativity of the diagram \eqref{FuLoSptoSt.173} we apply
Lemma~\ref{L:cech_pushforward} again. Here the pull-back $\ev^* P$ of a $\PU$ 
bundle $P \longrightarrow M$ to $\bbS\times \cL M$ is trivialized over the
cover $\bbS \times \Gamma_j$ by sections $s_j : \bbS \times \Gamma_j
\longrightarrow \ev^* P.$ The Dixmier-Douady class of $\ev^* P$ in $H^3(\bbS
\times \cL M; \bbZ)$ is then represented as a $\UU(1)$-\v Cech 2-cocycle by
\[
	u_{ijk} = \hat \tau_{ij} \hat \tau_{jk} \hat \tau_{ki} : \bbS \times \Gamma_{ijk} \longrightarrow \UU(1)
\]
where $\hat \tau_{ij} : \bbS\times \Gamma_{ij} \longrightarrow \UU(H)$ are
arbitrary unitary lifts of the $\PU$ difference classes
\begin{equation}
\tau_{ij} : \bbS \times \Gamma_{ij} \longrightarrow \PU,\ s_i = \tau_{ij}s_j,
\label{E:transgr_sec_diff}\end{equation}
and transgression of $[P] \in H^3(M; \bbZ)$ is represented in $H^2(\cL M ;
\bbZ)$ by the winding number cocycle $w_{ijk} : \Gamma_{ijk} \longrightarrow
\bbZ$ of $u_{ijk}(\theta,l)u^{-1}_{ijk}(0,l)$.

On the other hand, each section $s_j$ is related to the parallel lift of the
initial point of the section via
\begin{equation}
s_j(\theta,l) = h_j(\theta,l) \wt l_{s_j(0,l)},\
h_j : [0,2\pi]\times \Gamma_j \longrightarrow \PU,\ h(0,l) \equiv 1,
	\label{E:transgr_sec_par}
\end{equation}
where $h^{-1}_j(2\pi,l) \in \PU$ is the holonomy of $l$ with initial point $s_j(0,l)
\in P_{l(0)}.$ The $h_j$ also admit unitary lifts $\hat h_j : [0,2\pi] \times
\Gamma_i \longrightarrow \UU(H)$, and it follows from the construction of the
holonomy bundle $D \longrightarrow \cL M$ that $l \mapsto \hat h^{-1}_j(2\pi,l)$ is
a trivializing section of $D$ over $\Gamma_j$. Hence the Chern class of $D$
is represented by the $\UU(1)$-\v Cech 1-cocycle
\[
\mu_{ij} : \Gamma_{ij} \longrightarrow \UU(1),\
	\hat h^{-1}_i(2\pi,l) = \mu_{ij}(l) \hat h^{-1}_j(2\pi,l).
\]

It follows from \eqref{E:transgr_sec_diff} and \eqref{E:transgr_sec_par} that
\begin{equation}
	h_i(\theta,l) \tau_{ij}(0,l) = \tau_{ij}(\theta,l) h_j(\theta,l).
	\label{E:transgr_hol_reln}
\end{equation}
Since $\tau_{ij}(0,l) = \tau_{ij}(2\pi,l)$ this expresses the fact that the
holonomies at the different initial points are congugate: $h_i(2\pi,l) =
\tau_{ij}(0,l)h_j(2\pi,l)\tau^{-1}_{ij}(0,l).$ In any case, it follows from
\eqref{E:transgr_hol_reln} that the unitary lifts of the $h_j$ and the
$\tau_{ij}$ are related by
\[
\begin{gathered}
\sigma_{ij}(\theta,l) \hat h_i(\theta,l) \hat \tau_{ij}(0,l) =
\hat \tau_{ij}(\theta,l)\hat  h_j(\theta,l), \\
\sigma_{ij} : [0,2\pi]\times \Gamma_{ij} \longrightarrow \UU(1), \\
\sigma_{ij}(0,l) \equiv 1, \ \sigma_{ij}(2\pi,l) =
\hat \tau_{ij}(0,l)\hat h^{-1}_i(2\pi,l)\hat \tau_{ij}(0,l) \hat h_j(2\pi), \\
u_{ijk}(\theta,l)u^{-1}_{ijk}(0,l) = \sigma_{ij}\sigma_{jk}\sigma_{ki}.
\end{gathered}
\]
Note that since $D$ is the quotient of an equivariant bundle by the adjoint
action of $\PU$, it follows that $\sigma_{ij}(2\pi) : \Gamma_{ij}
\longrightarrow \UU(1)$ represents the Chern class of $D$ as well as
$\mu_{ij}.$ There are normalized logarithms $\eta_{ij} : [0,2\pi] \times \Gamma_{ij}
\longrightarrow \bbR$ such that $\sigma_{ij}(\theta,l) = \exp(2\pi i
\eta_{ij}(\theta,l))$, and the winding number cocycle is given by
\[
	w_{ij}(l) = \eta_{ij}(2\pi,l) + \eta_{jk}(2\pi,l) + \eta_{ki}(2\pi,l) \in \bbZ
\]
which also represents the $\check C^2(\cL M; \bbZ)$ lift of
$\sigma_{ij}(2\pi)$, and hence the Chern class of $D$.
\end{proof}

Finally, we consider the strengthening of the equivalence relation on fusive
bundles.

\begin{proposition}\label{P:FCB_fusive_isomorphic}
Two fusive bundles represent the same class in $H^2_{\fus}(\cL M)$ if and only
if they are fusive isomorphic, i.e.\ there is a lithe isomorphism between them
which intertwines fusion and the actions of the full reparametrization
semigroup.
\end{proposition}

\begin{proof} A fusive isomorphism is a lithe-fusion isomorphism, so it
suffices to assume that $[D_1] = [D_2] \in H^2_{\fus}(\cL M)$ and show that
$D = D_1^{-1}\otimes D_2$ has a global section which is lithe and invariant
under fusion and the reparametrization action.

By the injectivity of regression, the bundle gerbe defined by $D$ is simplicially
trivial, so that
\[
\psi^* D \cong \pi_1^* J \otimes \pi_2^* J^{-1},\Mwith J \longrightarrow \cI M
\] 
trivial.

The lithe, restricted-reparametrization-invariant connection on $D$ induces
connections on each of the bundles $J_{i,j}$ in \eqref{E:local_J}, which are
invariant with respect to reparametrizations of paths. Moreover by fusion
invariance the connection on $D$ is recovered over
$(\ev(0),\ev(\pi))^{-1}(B_i\times B_j) \subset \cL M$ as the tensor product
connection on $J_{i,j}\boxtimes J_{i,j}^{-1}.$ In view of the fusion condition,
the connections induced on the $D_{ik,jl}$ in \eqref{E:transition_bundle} are
trivial over the canonical trivialization over 3-fold intersections and so
correspond to real 1-forms on the double intersections $B_{ij}\times B_{kl}$
which form a \v Cech class valued in in 1-forms. Thus there exist smooth
1-forms $\mu_{ij}$ on the $B_i\times B_j$ with these as \v Cech
boundary. Shifting the connection on $J_{ij}$ by the $\mu_{ij}$ gives a
consistent global connection on $J$ which still induces the original connection
on $D$ and which remains restricted reparameterization-invariant.

Now $J$ has been arranged to be trivial and over the constant paths $M \subset
\cI M$ and has a section $u : M \longrightarrow J.$ Consider the retraction
of $\cI M$ onto the constant paths at the initial points via
\[
\cI M\times [0,1] \longrightarrow \cI M,\
	(\gamma,t) \mapsto \gamma_t,\ \gamma_t(s) = \gamma(ts).
\]
We proceed to show that the parallel translation of $u$ along the inverse of this
retraction results in a global restricted reparametrization invariant section of $J$
over $\cIE M.$ To see this, first observe that the retraction of a
reparametrization of a given path is equivalent to the action by a curve of
reparametrizations applied to the retraction of the path. Specifically, if $r
\in \Rp^+([0,2\pi])$ and $\gamma \in \cI M$, then
\[
(r\cdot \gamma)_t = r_t \cdot \gamma_t,\
	\Rp^+([0,2\pi]) \ni r_t(s) = \begin{cases} r(ts)/t & 0 < t \leq 1, \\
r(0) & t = 0.\end{cases}
\]
Note that the curve of reparametrizations is smooth down to $0$ since $r$
is smooth and fixes $0.$ If $u(\gamma_t)$ is the parallel lift of
$u(\gamma_0)$ along the curve $t \longmapsto \gamma_t$ in $\cI M$, it follows from
equivariance of the connection that $t\longmapsto r_t\cdot u(\gamma_t)$
is the (necessarily unique) parallel lift of $u(\gamma_0)$ over $t \mapsto r_t
\cdot \gamma_t = (r \cdot \gamma)_t$, and we conclude that $u$, evaluated at $t
= 1$, is indeed a global restricted-reparametrization-equivariant section.

Thus, $T=\pi_1^* u \otimes \pi_2^* u^{-1}$ is a global piecewise lithe
section of $D$ which is invariant under fusion and the action of
$\Rp_{\set{0,\pi}}(\bbS).$ In fact this section must be invariant under the
full reparameterization semigroup. The definition of a fusive bundle
includes the requirement of local lithe triviality, simultaneously, of the
reparameterization action. So each point is contained in a
reparameterization-invariant open set over which there is such a
trivialization. Over such an open set there is therefore a fusive function
$f$ which maps $T$ to section which is invariant under all
reparameterizations. Since $T$ itself is invariant under the restricted
group, $f$ must be invariant under this subgroup. The Lie algebra of
this group consists of the vector fields vanishing at $0$ and $\pi$ so is
dense in all vector fields on the circle in the $L^\infty$ topology. Thus
Lemma~\ref{FuLoSptoSt.456} implies that it is reparameterization-invariant
and hence $T$ itself is invariant under reparameterization.
\end{proof}

\begin{lemma}\label{FuLoSptoSt.374} For a fusive circle bundle the
collection of reparameterization actions compatible with the given
fusion isomorphism is affine.
\end{lemma}

\begin{proof} By assumption the bundle $D$ satisfies FB.i)--FB.vi) but
initially we ignore the fusion conditions. The reparameterization action $A$
gives a $\UU(1)$-equivariant isomorphism $A_r:D_{l}\longrightarrow
D_{l\circ r}$ for each $r\in\Dff^+(\bbS)$ and each $l\in\cL M.$ Thus if
$\tilde A$ is a second such action then
\begin{equation}
\begin{gathered}
A_r^{-1}\tilde A_r:D_{l}\longrightarrow D_{l}\text{ defines}\\
g:\Dff^+(\bbS)\times\cL M\longrightarrow \UU(1)\Mst\\
\tilde A_r=A_r g(r,l)\Mon D_l.
\end{gathered}
\label{FuLoSptoSt.375}\end{equation}
If $r_1,$ $r_2\in\Dff^+(\bbS)$ then at $l\in\cL M,$ 
\begin{equation}
\begin{gathered}
\tilde A_{r_1}\tilde A_{r_2}=\tilde A_{r_1}A_{r_2}g(r_2,l)
=A_{r_1}A_{r_2}g(r_1,l\circ r_2)g(r_2,l)\\
\Longrightarrow g(r_1r_2,l)=g(r_1,l\circ r_2)g(r_2,l)
\end{gathered}
\label{FuLoSptoSt.376}\end{equation}
and conversely this condition ensures that $\tilde A$ is a
reparameterization action if $A$ is one.

Thus $g(r,l)$ is a twisted character on $\Dff^+(\bbS)\times\cL M,$ in
particular this fixes a class in $H^1(\Dff^+(\bbS);\bbZ)+H^1(\cL M;\bbZ).$
Restricted to the identity $g\equiv 1$ so in fact the class descends to
$H^1(\Dff^+(\bbS); \bbZ).$

To compute this class, consider the restriction of $g$ to a constant loop, a
fixed point of the reparameterization group. Then \eqref{FuLoSptoSt.376}
reduces to the condition that $g$ be a character, a homomorphism $\tilde
g:\Dff^+(\bbS)\longrightarrow \UU(1).$ However all such lithe maps are
trivial. Indeed, the differential $f$ of $\tilde g$ at the identity is a linear map on
the tangent space $\CI(\bbS; T\bbS) \cong \CI(\bbS)$ of $\Dff^+(\bbS)$ which vanishes on
commutators. That is, using the assumed litheness of the action,
$f\in\CI(\bbS)$ satisfies
\begin{equation}
\int_{\bbS}f(ab'-a'b)d\theta=0\ \forall\ a,b\in\CI(\bbS).
\label{FuLoSptoSt.377}\end{equation}
Taking $b=1$ it follows that $f'=0$ and $f$ is constant. Then
\eqref{FuLoSptoSt.377} reduces to 
\begin{equation*}
\int_{\bbS}ab'=0\ \forall\ a,b\in\CI(\UU(1)),
\label{FuLoSptoSt.378}\end{equation*}
unless $f\equiv0.$ However, the latter identity cannot hold in general, for
example $a$ can be chosen to be a positive function of compact support in a
small interval where $b$ is linear.

The differential of $\tilde g$ therefore vanishes, and by exponentiating
(parameter dependent) vector fields it follows that $\tilde g\equiv1$ as claimed.

Since the cohomology class of $g$ is zero it has a global
logarithm, $g=\exp(2\pi i\eta)$ which is unique when normalized to vanish
at $\Id\times\cL M.$ The multiplicative condition \eqref{FuLoSptoSt.376}
for $g$ then reduces to the corresponding additive condition (given the
connectedness of the space) 
\begin{equation*}
\eta(r_1r_2,l)=\eta(r_1,l\circ r_2)+\eta(r_2,l).
\label{FuLoSptoSt.379}\end{equation*}
Thus $g_t=\exp(2\pi it\eta),$ $t\in[0,1]$ defines a 1-parameter family of
reparameterization actions connecting the two actions and all actions form
an affine space modelled on this linear space of functions.

The condition that $A$ and $\tilde A$ are both compatible with a fixed
fusion structure implies precisely that $g$ itself is fusive. Since the
logarithm then satisfies the corresponding linear condition it follows that
the reparamterization actions compatible with a fixed fusion structure also
form an affine space modelled on the lithe functions satisfying
\eqref{FuLoSptoSt.379} and the corresponding fusion condition.
\end{proof}

\section{String structures}\label{SectSG} In this section we briefly review
the notion of string structures and summarize the conclusions of Redden
\cite{Redden2011}; in particular how these are related to the exact sequence
\eqref{FuLoSptoSt.28}.

By definition, a string group is a topological group $\String(n)$ with a
surjective homomorphism
\begin{equation}
	\String(n)\longrightarrow \Spin(n)
	\label{E:string_to_spin}
\end{equation}
the kernel of which is a $K(\bbZ,2)$ and such that this bundle is
classified by a map $f : \Spin(n)\longrightarrow BK(\bbZ,2) = K(\bbZ,3)$
representing the canonical generator $\gamma_\mathrm{can} \in
H^3(\Spin,\bbZ) \cong \bbZ.$ We restrict to $n \geq 5$ so that $\Spin$ is
well-behaved. Note that $\String(n)$ cannot be a finite dimensional Lie
group since no such group can contain a $K(\bbZ,2)$ as a subgroup.

From the defintion of $\String(n)$ it follows that the induced map
$B\String(n)\longrightarrow B\Spin(n)$ is equivalent to the homotopy fiber
of $B\Spin(n)\longrightarrow BK(\bbZ,3) = K(\bbZ,4).$ Thus a string
structure exists on $M$ if and only if the induced class, which is $\ha
p_1(M) \in H^4(M,\bbZ)$ vanishes, and then homotopy classes of maps to
$B\String(n)$ form a torsor over $H^3(M,\bbZ).$

On the other hand, a string structure induces a $K(\bbZ,2)$ bundle over
$F=F_\Spin$ with restriction to each fiber equivalent to the $K(\bbZ,2)$
bundle \eqref{E:string_to_spin}; such bundles on $F$ are in
bijective correspondence with the set 
\begin{equation*}
\strC(F)=\set{\gamma \in H^3(F, \bbZ);i^\ast \gamma = \gamma_\mathrm{can} \in
H^3(\Spin,\bbZ)}.
\label{FuLoSptoSt.177}\end{equation*}
Conversely an element of $\strC(F)$ corresponds to an equivalence class of
string structures since there is an exact sequence
\begin{equation}
\xymatrix{
0 \ar[r]& H^3(M,\bbZ)\ar[r]& H^3(F,\bbZ) \ar[r]& H^3(\Spin,\bbZ) \ar[r]&
\ha p_1(M) \bbZ \ar[r]&  0
}
\label{E:exseq}\end{equation}
which follows directly from the Leray-Serre spectral sequence for $F$ and the fact
that $\Spin$ is 2-connected.

In particular, $\strC(F)$ is empty unless $\ha p_1(M) = 0$, in which case
it is an $H^3(M,\bbZ)$ torsor and so equivalence classes of string structures
are in bijection with $\strC(F).$

\section{Fusive loop-spin structures}\label{SectFuLoSp}

From now on we assume that $M$ is a spin manifold, with a given lift of the
orthonormal frame bundle to a spin-oriented frame bundle $F.$ By a
{\em loop-spin structure} we mean an extension of the $\cL\Spin$ bundle $\cL F$
to a principal bundle with structure group $E\cL\Spin,$ the basic central
extension, with the additional properties listed below. The total space of such an extended
bundle is a circle bundle, $D,$ over $\cL F$ and we state conditions in
terms of this.

\begin{enumerate}
[{\normalfont LS.i)}]
\item \label{LSfus}(Fusive bundle) First, we require that $D$ be a fusive circle bundle over $\cL F$
in the sense of \S\ref{SectFuCiBu}. Thus it is lithe, has an action of the
reparameterization semigroup which is the trivial action in appropriate
local trivializations, has a fusion isomorphism and satisfies the
compatibility conditions between these.
\item \label{LSaction} (Principal action)
The principal bundle structure on the total space of $D$ corresponds to a twisted
equivariance condition covering the action of $\cL\Spin$ on $\cL F,$ and we
therefore demand that there be a lithe circle bundle isomorphism (`multiplication')
\begin{equation}
\begin{gathered}
M:\pi_1^* E \otimes \pi_2^* D \longrightarrow m^* D \Mover \cL\Spin\times\cL F,\\
m:\cL\Spin\times\cL F\ni (\gamma ,l)\longmapsto \gamma l\in\cL F.
\end{gathered}
\label{FuLoSptoSt.178}\end{equation}
To induce a group action this needs further to satisfy the associativity
condition that the pointwise diagram involving the product on $E\cL\Spin$ 
\begin{equation}
\xymatrix{
E_{g_1}\otimes E_{g_2}\otimes D_l \ar[r]^-{M\otimes \Id} \ar[d]_{\Id \otimes M} & E_{g_1\,g_2} \otimes D_l \ar[d]^M \\
E_{g_1} \otimes D_{g_2\,l} \ar[r]^M & D_{g_1\,g_2\,l}
}
\label{FuLoSptoSt.194}\end{equation}
commute for each $g_1,$ $g_2\in\cL\Spin$ and $l\in\cL F.$
\item \label{LScompat} (Compatibility)
It is further required that $M$ be compatible with the fusion isomorphism
in the sense that the diagram
\begin{equation}
\xymatrix{
E_{g_{12}}\otimes E_{g_{23}} \otimes D_{l_{12}}\otimes D_{l_{23}}
\ar[r]^-{\Phi\otimes \Phi} \ar[d]_{M\otimes M}
& E_{g_{13}} \otimes D_{l_{13}} \ar[d]^M \\
D_{g_{12}\,l_{12}}\otimes D_{g_{23}\,l_{23}} \ar[r]^-{\Phi} 
& D_{g_{13}\,l_{13}}
}
\label{FuLoSptoSt.200}\end{equation}
commutes for $(\gamma_1,\gamma_2,\gamma_3)\in\cI^{[3]}F,$ with $l_{ij}=\psi(\gamma_i,\gamma_j)$ and
$(h_1,h_2,h_3)\in\cI\Spin$ with $g_{ij}=\psi(h_i,h_j).$

There is compatibility condition between the reparameterization semigroup
action $A$ and $M$ corresponding to the known action $A$ of $\Rp$ on
$E\cL\Spin$ and so reducing to the commutativity of
\begin{equation}
\xymatrix{
E_{g\circ r}\otimes D_{l\circ r}\ar[r]^-{A\otimes A} \ar[d]_M & E_{g}\otimes D_l \ar[d]^M \\
D_{(gl)\circ r}\ar[r]^-A &D_{gl}
}
\label{FuLoSptoSt.198}\end{equation}
for each $r \in\Rp(\bbS)$ and $l\in\cL F.$
\end{enumerate}

Two fusive loop-spin structures are equivalent if and only if the two
circle bundles are lithe-isomorphic over $\cL F$ by an isomorphism that
intertwines the structure isomorphism $M$ and $\Phi.$

We denote the set of equivalence classes of fusive loop-spin structures by
\[
	\strC_{\fus}(F).
\]

\begin{proposition}\label{FuLoSptoSt.25} If $D_i,$ $i=1,$ $2,$ are fusion
  loop-spin structures associated to the same spin structure on $M$ then
  $D_1\otimes D_2^{-1}$ is the pull-back of a fusive circle bundle, $K,$
  over $\cL M,$ and $D_1$ and $D_2$ are equivalent if and only if $K$ is
  lithe-fusion trivial. Conversely, for any fusive circle bundle $K$ on
  $\cLE M,$ and fusion loop-spin bundle $D$ the product $D \otimes K$ has a
  fusion loop-spin structure.
\end{proposition}
\noindent It follows that, if non-empty, the collection $\strC_{\fus}(F)$
of equivalence classes of fusive loop-spin structures is a torsor over
$H^3(M;\bbZ).$

\begin{proof} If $D_i,$ $i=1,2$ are fusive loop-spin structures in the
sense described above, then
\begin{equation}
\tilde K=D_1\otimes D_2^{-1}
\label{FuLoSptoSt.27}\end{equation}
is a fusive circle bundle over $\cL F.$ The tensor product of the two
multiplicative isomorphisms gives an untwisted lithe isomorphism 
\begin{equation}
\tilde M:m^*\tilde K\longrightarrow \pi_2^*\tilde K\Mover \cL\Spin\times\cL F.
\label{FuLoSptoSt.203}\end{equation}
The associativity condtions \eqref{FuLoSptoSt.194} ensure that this gives
an action of $\cL\Spin$ on $\tilde K$ which covers the action on $\cL F$
and hence is free. Thus $\tilde K$ descends to a circle bundle $K$ over
$\cL M.$ By restricting to local sections this bundle can be seen to be
lithe and the consistency conditions \eqref{FuLoSptoSt.200} ensures that the
induced fusion isomorphism on $\tilde K$ descends to a fusion isomorphism
on $K.$ Thus $K$ is indeed a fusive circle bundle over $\cL M$ and hence
defines an element of $H^2_{\fus}(\cL M)=H^3(M;\bbZ)$ the vanishing of
which is equivalent to fusive triviality.

An isomorphism between two fusive loop-spin structures, $D_1$ and $D_2,$ is a
fusive section of $\tilde K$ with equivariance properties corresponding to the
intertwining conditions. In particular it descends to a lithe section of $K$
which commutes with the induced fusion isomorphism and so implies its fusive
triviality. Conversely, a fusive section of $K$ lifts to generate an
isomorphism of the two loop-spin structures.  The map to $H^3(M; \bbZ)$ is
surjective since if $K$ is a fusive circle bundle over $\cL M$ and $D_2$ is a
loop-spin structure then $D_1=K\otimes D_2$ has a natural loop-spin structure
giving $K$ as difference bundle.
\end{proof}

Next we show that a fusive loop-spin structure defines an element of
the set $\strC(F)$ in \eqref{FuLoSptoSt.177} and that this induces a map 
\begin{equation}
\strC_{\fus}(F)\longrightarrow \strC(F).
\label{FuLoSptoSt.204}\end{equation}
%

\begin{proposition} The Dixmier-Douady class in $H^3(F; \bbZ)$ assigned to a loop-spin
structure by Proposition~\ref{FuLoSptoSt.380} lies in $\strC(F)$ and this
induces a map \eqref{FuLoSptoSt.204} consistent with the torsor structure
of Proposition~\ref{FuLoSptoSt.25}. 
\end{proposition}

\begin{proof} To show that the element of $H^3(F;\bbZ)$ assigned to the
bundle gerbe associated to a loop-spin structure lies in
$\strC(F)$ it is only necessary to show that restricted to any one fibre of
$F$ it is a basic gerbe for $\Spin,$ i.e.\ the 3-class is the chosen
generator of $H^3(\Spin).$

Thus consider a fiber $F_m \subset F.$ The pull-back of the bundle gerbe over
$F^2$ to $F^2_m$ has fiber space consisting of paths with end-points in $F_m.$
This contains the sub-gerbe with fiber $\cI F_m$ consisting of paths entirely
contained in the fiber. The inclusion map is a bundle gerbe map and hence the
gerbes have the same Dixmier-Douady invariants. On the other hand, restricted
to the fiber loops $\cL_{\fib} F_m$, the fusive bundle $D$ is isomorphic to the
circle bundle $E\cL\Spin \longrightarrow \cL \Spin,$ which has Dixmier-Douady
class the chosen generator of $H^3(\Spin;\bbZ).$

As discussed in Proposition~\ref{FuLoSptoSt.25} the $H^3(M;\bbZ)$-torsor
structure on fusive loop-spin structures is generated by the tensor
difference. Tensoring with a fusive circle bundle pulled back from $\cL M$
is equivalent to taking the gerbe product with the pull-back of the
corresponding gerbe over $M$ and so shifts the Dixmier-Douady class by the
pull-back into $H^3(F;\bbZ)$ of the corresponding element of $H^3(M;\bbZ).$
Thus the induced map \eqref{FuLoSptoSt.204} commutes with the torsor actions.
\end{proof}

\section{The holonomy bundle for a string class}\label{RegCF}

If $\alpha \in\strC(F)$ is a string class then there is a principal $\PU$
bundle $Q_F=Q_F(\alpha),$ smooth in the norm topology, over $F$ with
Dixmier-Douady class $\alpha.$ Although the class is invariant under the
action of $\Spin$ on $F$ the bundle cannot be equivariant with respect to
this free action, since it would then descend to $M.$ Nevertheless, we need
to construct such a $\Spin$ action on the holonomy bundle of $Q_F$ as a
preliminary step to the constuction of the full loop spin action. 

\begin{proposition}\label{FuLoSptoSt.447} To each class $\alpha\in\strC(F)$
there corresponds an equivalence class in $H^2_{\fus}(\cL F)$ of fusive
circle bundles over $\cL F,$ containing the holonomy bundle of
$Q_F(\alpha),$ and each of these bundles can be given an equivariant,
lithe, fusion- and reparameterization-invariant action of $\Spin$ covering
the action by constant loops on $\cL F.$
\end{proposition}

\begin{proof} Given the results of \S\ref{SectFuCiBu}, the only issue
remaining is the construction of a compatible equivariant $\Spin$
action. To do so we pass from $F$ to the product $F^2$ and the quotient
$$
\AUT(F)=F^2/\Spin
$$
by the diagonal $\Spin$ action. Since $F^2$ is a principal $\Spin$ bundle
over $\AUT(F)$ there is an isomorphism given by pull back
\begin{equation}
H^3(\AUT(F);\bbZ)\longrightarrow \{\beta \in H^3(F^2;\bbZ)\emph{ fiber trivial}\}.
\label{FuLoSptoSt.339}\end{equation}
In particular each difference class $\pi_1^*\alpha -\pi_2^*\alpha$ is fiber
trivial, so is the pull-back of a unique class $\beta (\alpha )\in
H^3(\AUT(F);\bbZ).$

Thus, rather than $Q_F,$ we may consider a smooth $\PU$ bundle, $Q,$ over
$\AUT(F)$ with Dixmier-Douady class $\beta(\alpha).$ With each path in $F$
there is an associated path in $F^2$ given by the constant path at the
initial point in the left factor and the path itself in the right
factor. This projects to $\AUT(F)$ giving a map
\begin{equation}
\cIE F\ni\mu \longmapsto (\mu(0),\mu)\in\cIE F^2\longrightarrow \cIE\AUT(F).
\label{FuLoSptoSt.340}\end{equation}
The range consists precisely of the smooth paths in $\AUT(F)$ which cover a
path in $M^2$ which is constant in the first factor, starts at the
diagonal, and has initial value the identity automorphism of the
fibre of $F$ at the initial point. Moreover, \eqref{FuLoSptoSt.340} is a
principal bundle over the range for the action of $\Spin$ on paths and so
\eqref{FuLoSptoSt.340} gives an embedding of paths which extends to an
embedding of loops:
\begin{equation}
\cIE F/\Spin\longrightarrow \cIE\AUT(F),\ \cLE F/\Spin\longrightarrow \cLE\AUT(F).
\label{FuLoSptoSt.341}\end{equation}
This embedding maps smooth loops to smooth loops and the fusion product in
$F$ to the corresponding product for $\AUT(F).$ It also identifies
restricted reparameterization for loops in $F$ with restricted
reparameterization for loops in $\AUT(F).$ However the dependence of the
embedding on the initial point of the path in $F$ means that it does not map
the rotation action on loops to the corresponding rotation action.

The discussion in \S\ref{SectFuCiBu} applies to $Q$ as a bundle over $\AUT(F).$ The
restriction of $\AUT(F)$ to the diagonal in $M^2$ is the automorphism
bundle of $F$ itself, i.e.\ automorphisms of fixed fibers of $F.$
Restricting $Q$ to the identity section over this submanifold and then
pulling back to $\cL F/\Spin,$ via the embedding \eqref{FuLoSptoSt.341} and
the map to the initial point, parallel translation of $Q$ around the image
loop in $\AUT(F)$ and comparison to the initial point gives the holonomy map
and bundle as in \eqref{FuLoSptoSt.170}
\begin{equation}
\begin{gathered}
H:Q\times_{\ev(0)}(\cL F/\Spin)\longrightarrow \PU\\
\tilde D(A)=H^*\UU,\  D(A)=\pi^*\tilde D(A)/\PU\longrightarrow \cL F,\
\pi:\cL F\longrightarrow \cL F/\Spin.
\end{gathered}
\label{13.6.2013.4}\end{equation}

Thus $D(A)$ is a circle bundle over $\cL F$ which is restricted fusive, in
the sense that only the semigroup of reparameterizations fixing $0\in\bbS$
acts, and has a compatible equivariant $\Spin$ action covering the action
by constant loops. On the other hand we know that the gerbe induced by
$D(A)$ over $F^2$ has Dixmier-Douady class $\pi_1^*\alpha -\pi_2^*\alpha$
since restricted to pointed paths it has class $\alpha.$

The discussion in \S\ref{SectFuCiBu} shows that $D(A)$ and
$D,$ corresponding to the same class $\alpha \in\strC(F)$ are fusive
isomorphic. This allows the equivariant $\Spin$ structure on $D(A)$ to be
transferred to $D.$ However, as noted above, this need not be compatible
with the full reparameterization action. Rather, the action of $\Spin$ on
$D$ induces a multiplicative map of $\Spin$ into the affine space of
fusion-compatible reparameterization actions, as discussed in
Lemma~\ref{FuLoSptoSt.374}. This allows the reparameterization action to be
averaged over $\Spin$ to produce a compatible (and fusion-compatible)
action. Thus if $s\in\Spin$ then $s^*A=A+a_s,$ where
$a_\bullet\in\CI_{\fus}(\Spin\times\cL F)$ is in the linear space of
twisted characters discussed in Lemma~\ref{FuLoSptoSt.374} and satisfies
$$
a_{st}=a_s+s^*a_t,\ s,\ t\in\Spin.
$$
Then setting $\bar a=\int_{\Spin}a_sds$ the action $A+\bar a$ is compatible
with $\Spin.$
\end{proof}

The fusive circle bundle $D=D_\alpha,$ with its compabible $\Spin$
structure will be called the \emph{holonomy bundle of the string class} $\alpha.$ 

Below we construct a loop-spin structure on $D.$ First consider its
restriction to fiber loops in $F.$ Each $l\in\cL_{\fib}F$ is determined by
its initial point in $F$ and the `difference loop' giving the shift in the
fiber and conversely each such loop is given by shifting a constant
loop. Thus there is an isomorphism
\begin{equation}
\cL_{\fib}F\longrightarrow F\times\dcL\Spin.
\label{FuLoSptoSt.343}\end{equation}
The $\Spin$ action on the left corresponds to the principal $\Spin$ action
on $F$ and the adjoint action of $\Spin$ on $\dcL\Spin.$ Thus the quotient
by this free action is a bundle of groups over $M,$ modelled on $\dcL\Spin$
and with structure group $\Spin.$ Since the central extension of
$\dcL\Spin$ is equivariant for the conjugation action, the pull-back of $E$
is a well-defined bundle over $\cL_{\fib}F/\Spin.$ 

\begin{proposition}\label{FuLoSptoSt.205} The holonomy bundle $D$ of a
string class, given by \eqref{13.6.2013.4}, restricted to fiber paths is
$\Spin$-equivariant and fusive isomorphic to the pull-back of the basic
central extension of $E\cL\Spin.$
\end{proposition}

\begin{proof} The Dixmier-Douady invariant of the bundle gerbe induced by
$D$ is an element of
\begin{equation}
H^3(F^{[2]};\bbZ)=H^3(F;\bbZ)\oplus H^3(\Spin;\bbZ).
\label{FuLoSptoSt.209}\end{equation}
The naturality properties of the invariant show that the first term is
the Dixmier-Douady invariant of the restriction to $F\longrightarrow
F^{[2]}$ as the fiber diagonal. Over this submanifold there is a natural
section of $\cI_{\fib}F$ given by the constant path at each point. This
gives a trivial subgerbe and shows that the restriction to this subspace is
itself trivial. On the other hand, restricted to one fiber of $\Spin$ it is
the holonomy of the restricted $\PU$ bundle which by definition corresponds
to the chosen generator of $H^3(\Spin;\bbZ).$ Thus the gerbes induced by
$E$ and $D$ over $F^{[2]}$ do have the same Dixmier-Douady
invariants. Moreover, both are lifts of bundle gerbes over $M\times\Spin$
with the same 3-class so Proposition~\ref{FuLoSptoSt.380} gives a fusive isomorphism
between them there, and hence lifted to $F^{[2]}$ where it intertwines the
$\Spin$ actions.
\end{proof}

\section{Blips}\label{Blipping}

Although most of the discussion of loops has been relegated to an appendix,
we describe here a crucial component of our construction of fusive
loop-spin structures in the next section. This is a map from smooth loops
in the spin frame bundle to special piecewise smooth loops that we call
`blips'.

The basic form of the blip construction can be carried out on free paths in
$F.$ We associate to a given path $\gamma\in\cI F$ two paths. The first,
horizontal path, is determined by three conditions. It has the same projection
into $M$ as $\gamma$ and it has the same initial point in $F$ as $\gamma$
but as a section of the pull-back of $F$ to the base curve it is covariant
constant with respect to the pull-back of the Levi-Civita connection:
\begin{equation}
h_{\gamma}\in\cI F,\
\pi(h_\gamma)=\pi(\gamma),\ h_{\gamma}(0)=\gamma(0),\ h_{\gamma}^*\nabla=0.
\label{FuLoSptoSt.263}\end{equation}
Thus, $h_\gamma$ is obtained from $\gamma$ by replacing the given section
of $F$ over the projection of the curve into $M$ by the covariant constant,
which is to say horizontal, section of $F$ with the same initial point as $\gamma.$ 

Since $\gamma$ and $h_\gamma$ cover the same path in $M,$ there is a
well-defined path in $\Spin$ shifting $h_\gamma$ to $\gamma:$  
\begin{equation*}
	\gamma = s_\gamma\, h_\gamma,\ s_\gamma\in\cI\Spin,\ s_\gamma (0)=\Id.
\label{FuLoSptoSt.264}\end{equation*}
Then the `vertical' path associated to $\gamma$ is in the fiber of $F$ through the
end-point of $\gamma$ and is given by applying this path to the constant
path 
\begin{equation*}
v_{\gamma}\in\cI_{\fib}F,\ v_\gamma = s_\gamma h_\gamma(2\pi).
\label{FuLoSptoSt.265}\end{equation*}
Thus the initial point of $v_{\gamma}$ is the end-point of $h_\gamma$
 and we define the `blip' path associated to  $\gamma$ to be the join of these:
\[
\begin{gathered}
	\Bl : \cI F \longrightarrow \cI_{(\pi)}F, \\
	\Bl(\gamma) = j\bpns{h_\gamma, v_\gamma}.
\end{gathered}
\]
Observe that $\Bl$ is a bijection onto its image which consists of all
paths, with one `kink' at $\pi,$ which are covariant constant on $[0,\pi]$
and are confined to a fiber over $[\pi,2\pi];$ $\Bl(\gamma)$ has the same
initial and end points as $\gamma.$ So in particular the operation extends
to the fibre products over $\cI F\longrightarrow F^2:$
\begin{equation}
\Bl^{[j]}:\cI ^{[j]}F\longrightarrow \cI^{[j]}_{(\pi)}F,\ j=2,3.
\label{FuLoSptoSt.266}\end{equation}

There is an even simpler, fiber, form of the blip map, corresponding to the base
being a point, which can be written explicitly as 
\begin{equation}
\bl:\cI\Spin\longrightarrow \cI_{(\pi)}\Spin,\ \bl(\mu)=j(\mu(0),\mu),
\label{FuLoSptoSt.271}\end{equation}
and is just given by adjoining to a path in $\Spin$ an initial constant path with the same
initial value as the given path. Again it preserves endpoints and has the
same simplicial properties as $\Bl$ above, of which it is a special
case.

In terms of this map the relation between the blip construction and the
action of $\cI\Spin$ is
\begin{equation}
\Bl(\mu\gamma )=\bl(\mu)\Bl(\gamma ),\ \mu\in\cI\Spin,\ \gamma\in\cI F.
\label{FuLoSptoSt.284}\end{equation}

Similarly, there is a simple relationship between reparameterization of
paths and blips. If $r\in\Dff^+([0,2\pi])$ is an oriented diffeomorphism
then for any $\gamma\in\cI F$ 
\begin{equation}
h_\gamma\circ r=h_{\gamma\circ r},\ s_{\gamma }\circ r=s_{\gamma\circ r}.
\label{FuLoSptoSt.285}\end{equation}
Thus  
\begin{equation}
\Bl(\gamma \circ r)=\Bl(\gamma)\circ j(r,r)
\label{FuLoSptoSt.286}\end{equation}
where the reparametrization is rescaled and applied on each of the
intervals $[0,\pi]$ and $[\pi,2\pi].$ 

Since the paths, forming $\cL_{(2)}F,$ with kinks at $0$ and $\pi$ may be
identified with $\cI^{[2]}F$ by the fusion map, $\psi,$ the blip
construction extends to loops:
\begin{equation}
\begin{gathered}
\Bl:\cL_{(2)}F  \longrightarrow \cL_{(4)}F\\
\Bl(\psi(p_1,p_2))=\psi(\Bl^{[2]}(p_1,p_2)),\ (p_1,p_2)\in\cI^{[2]}F.
\end{gathered}
\label{FuLoSptoSt.267}\end{equation}
The resulting blip loop consists of a horizontal segments for $t\in
[0,\pi/2] \cup [3\pi/2, 2\pi]$ and vertical path for $t \in [\pi/2,
3\pi/2].$ Applied to a smooth loop the blip construction yields a loop
which is smooth, rather than only piecewise smooth, at $0$ and at $\pi,$ but
not in general at the points $\pi/2$ and $3\pi/2.$ Thus in fact if we `rotate' the loop
by $\pi/2$ it has only kink points at $0$ and $\pi:$
\begin{equation}
R(\pi/2)\Bl:\cL F\longrightarrow \cL_{(2)}F.
\label{FuLoSptoSt.268}\end{equation}

\begin{lemma}\label{FuLoSptoSt.269} The blip map preserves the fusion
relation.
%
%
\end{lemma}

\begin{proof} This is immediate from the definitions.
\end{proof}

From the discussion above, similar relationships between the blip construction on
loops and the $\cL\Spin$ and (restricted) reparameterization actions follow. Thus
\begin{equation}
\Bl(\mu l) = \bl(\mu)\Bl(l),\ l\in\cL F,\ \mu\in\cL\Spin.
\label{FuLoSptoSt.273}\end{equation}

The action of a constant loop in $\Spin$ commutes with the blip map. On the
other hand the action of an element of $\dcI\Spin,$ a loop with initial
value $\Id,$ is `compressed' into the fiber action on the vertical part of
the blipped loop. Thus \eqref{FuLoSptoSt.268} may be rewritten
\begin{equation}
\begin{gathered}
R(\pi/2)\Bl(\lambda)=\psi(V_\lambda,H_\lambda ),\ (V_\lambda ,H_\lambda
)\in\cI^{[2]}F,\ \lambda \in\cL F,\\
\lambda =\psi(\gamma _1,\gamma _2),\ V_\lambda =\psi(v_{\gamma
_1},v_{\gamma _2}),\ H_\lambda =\psi(rh_{\gamma _1},rh_{\gamma_2})
\end{gathered}
\label{FuLoSptoSt.287}\end{equation}
and then the action of $\phi\in\dcL\Spin$ is again through the vertical part
\begin{equation}
R(\pi/2)\Bl(\phi\gamma)=\psi(\phi V_\gamma,H_\gamma)
\label{FuLoSptoSt.288}\end{equation}
where on the right $\phi$ is identified as a path rather than a loop. The
significance of \eqref{FuLoSptoSt.288} is that  
\begin{equation}
(\phi V_\gamma ,V_\gamma ,H_\gamma )\in\cI^{[3]}F
\label{FuLoSptoSt.289}\end{equation}
with the first two paths in the same fiber. This allows the application of
the fusion isomorphism below.

There is a similar property for restricted reparameterization. If
$r\in\Dff_{\set{0,\pi}}^+(\bbS)$, then 
\begin{equation}
R(\pi/2)\Bl(\gamma\circ r)=\psi(V_\gamma\circ r',H_\gamma)
\label{FuLoSptoSt.290}\end{equation}
where $r'\in\Dff^+([0,2\pi])$ is simply $r$ transferred to the
interval. Then again there is a fusion triple 
\begin{equation}
(V_\gamma\circ r',V_\gamma ,H_\gamma)\in\cI^{[3]}F
\label{FuLoSptoSt.291}\end{equation}
with the first two paths in one fiber.

We also need to consider the regularity of the blip map. Consider a
coordinate patch $\Gamma(l,\epsilon)\subset\cI M.$ Over the geodesic ball
$B(l(0),\epsilon )\subset M$ the spin frame bundle is trivialized by the
section $f_0(x)$ obtained by choosing an element 
\begin{equation*}
f_0=f_0(l(0))\in F_{l(0)}
\label{FuLoSptoSt.299}\end{equation*}
which is then parallel-translated along radial geodesic to each $x\in
B(l(0),\epsilon ).$ Then each $l'\in\Gamma (l,\epsilon)$ has a determined
lift to $\cI F$ given by parallel translation of $f_0(l'(0))$ along $l'.$
This identifies the preimage in $\cI F$ as
\begin{equation}
\pi^{-1}\Gamma (l,\epsilon )=\Gamma (l,\epsilon )\times\cI\Spin
\label{FuLoSptoSt.292}\end{equation}
where on the right each element $l'$ is identified with $L(l',f_0),$ the
horizontal lift described above.

If $\gamma\in\cI F$ and $\pi(\gamma )\in\Gamma (l,\epsilon)$ then under
\eqref{FuLoSptoSt.292}  
\begin{equation}
\begin{gathered}
\gamma =S\cdot L(\pi(\gamma ),f_0),\ S\in\cI\Spin,\\
h_\gamma =S(0)L(\pi(\gamma),f_0),\ v_\gamma =S\circ S(0)^{-1}.
\end{gathered}
\label{FuLoSptoSt.293}\end{equation}
It follows directly that $\Bl$ extends to finite energy paths.

Next we show that the blip map is boundary-lithe on paths in the sense of
\S\ref{Sect.smooth}. The tangent bundle of $F$ decomposes according to the
Levi-Civita connection
\begin{equation}
	TF \cong VF\oplus HF.
\label{FuLoSptoSt.296}\end{equation}
Pulling this back, any section over a path be decomposed into horizontal
and vertical sections 
\begin{equation}
\CI([0,2\pi];\gamma^*TF)\ni u\longmapsto
\nu_Vu+\nu_Hu\in\CI([0,2\pi];\gamma^*VF)\oplus\CI([0,2\pi];\gamma^*HF) 
\label{FuLoSptoSt.297}\end{equation}
In addition the initial value of the vertical section can be extended
uniquely along the curve to be covariant constant to give a third `flat' section: 
\begin{equation}
\CI([0,2\pi];\gamma^*TF)\ni u\longmapsto \nu_Fu\in \CI([0,2\pi];\gamma^*VF).
\label{FuLoSptoSt.298}\end{equation}
The two vertical sections, $\nu_Vu$ and $\nu_Fu$ may be identified
naturally with maps into the Lie algebra of $\Spin.$ The derivative of
$\Bl$ can be evaluated from the definition above.

\begin{lemma}\label{FuLoSptoSt.294} The blip map is boundary-lithe in the sense that
its derivative is given in terms of the rescaling of the induced section in
\eqref{FuLoSptoSt.297}, \eqref{FuLoSptoSt.298} by
\begin{equation}
\begin{gathered}
d\Bl_\gamma: \CI([0,2\pi]; \gamma^\ast TF)\ni u\longmapsto \\
(\nu_Hu+\nu_Fu)(\cdot/2)\oplus\nu_Vu(\pi+\cdot/2)\in
\CI([0,\pi]; h_\gamma^\ast TF)\times\CI([\pi,2\pi];T_{\gamma(2\pi)}F).
\end{gathered}
\label{FuLoSptoSt.295}\end{equation}
\end{lemma}

In the discussion of the regularity of the loop-spin structure below, it is
important that the boundary part here, which is to say $\nu_Fu,$ appears
in both terms in terms in \eqref{FuLoSptoSt.295}, which can instead be
written as the sum of $\nu_Fu$ in both factors, the horizontal section in.
the first and the `reduced vertical section' $\nu_Vu-\nu_Fu$ in the second.

\section{Construction of fusive loop-spin structures}\label{Constr}

To complete the proof of the Main Theorem in the Introduction it remains
only to show the existence of a fusive loop-spin structure whenever $\ha
p_1=0.$ Under this condition Waldorf in \cite{Waldorf2012} shows the
existence of a loop-spin structure which satisfies the fusion condition,
but without lithe regularity. A circle bundle with the correct topological
type to correspond to a loop-spin structure is constructed in \S\ref{RegCF}
above.

\begin{theorem}\label{FuLoSptoSt.59} There is a fusive loop-spin structure
  on the loop space of a spin manifold of dimension $n\ge5$
  corresponding to each element of $\strC(F)$ and giving a right inverse to
  \eqref{FuLoSptoSt.204}.
\end{theorem}
\noindent The short exact sequence \eqref{FuLoSptoSt.28} shows that the
  vanishing of $\ha p_1$ is equivalent to $\strC(F)$ being non-empty.

\begin{proof} Consider a circle bundle $D$ as discussed in
\S\ref{RegCF}. The loop-spin structure will be constructed through the pull-back
of this bundle under the blip map \eqref{FuLoSptoSt.268}:
\begin{equation}
T=\Bl^*D,\ \Bl:\cL F\longrightarrow R(\pi/2)\cL_{(2)}F.
\label{FuLoSptoSt.274}\end{equation}
It follows directly from Lemma~\ref{FuLoSptoSt.269} that the fusion
structure on $D$ pulls back to a fusion structure $\Phi_T=(\Bl^{[3]})^*\Phi_D$ on $T.$ 

The fundamental property of $T$ is that it carries a multiplication
isomorphism. If $\gamma\in\cL F$ and $\lambda\in\cL\Spin$ then
$\lambda\gamma=\lambda (0)\phi\gamma$ where $\phi\in\dcL\Spin.$ Thus
\eqref{FuLoSptoSt.288} applies and
\begin{equation*}
\Bl(\lambda\gamma)=\lambda(0)R(\pi/2)\psi(\phi V_\gamma,H_\gamma).
\label{FuLoSptoSt.277}\end{equation*}
The fusion triple \eqref{FuLoSptoSt.289} gives the desired multiplication
identification, $M,$ as the long composite in
\begin{equation}
\xymatrix{T_{\lambda \gamma }=D_{\Bl(\gamma l)}\ar[r]^-M\ar[d]_{\lambda(0)}&
T_{\gamma}\otimes E_{\lambda}&D_{\Bl(\gamma)}\otimes
E_{\bl(\phi)}\ar[l]\\
D_{\Bl (\phi\gamma)}\ar[r]_-{A(\pi/2)}&D_{\psi(\phi V_\gamma,H_\gamma)}\ar[r]_-{\Phi_D}
&D_{\psi(\phi V_\gamma,V_\gamma)}
\otimes D_{\psi(V_\gamma,H_\gamma)}\ar[u]_{\lambda(0)^{-1}A(3\pi/2)\otimes \mu}. 
}
\label{FuLoSptoSt.239}\end{equation}
Here the map on the left is the $\Spin$ action on $D,$ the first map on
the lower line is given by the rotation action on $D$ and the
identification \eqref{FuLoSptoSt.288}, the second is the the fusion map on
$D$ corresponding to \eqref{FuLoSptoSt.289}. The map on the right is the
opposite rotation action on $D$ tensored by the identification $\mu$ of $D$
with $E$ over fiber loops, followed by the product decompositon of $E$
using the triviality of the loop $\psi(V,V).$ The map on the upper right is
reparameterization for $E$ -- excising the constant segment -- followed by
the action of the constant loop $\lambda(0)$ and the definition of $T.$
That this map is piecewise lithe is essentially a consequence of the fact that
it is well-defined.

We need to check that this has the properties required above of a loop-spin
structure as a circle bundle over $\cL F.$

First consider the associativity condition for $M.$ This involves comparing
the action of the product $\Lambda =\lambda_1\lambda_2$ of two loops in $\cL\Spin$
with the composite of the actions. The only complication here involves the
removal of the constant loop, which is necessary to make the paths in
\eqref{FuLoSptoSt.289} have the same end points. Thus, for the product loop 
the normalized loop involves conjugation:-
\begin{equation*}
\phi=\big(\lambda_2(0)^{-1}\phi_1\lambda_2(0)\big)\phi_2.
\label{FuLoSptoSt.280}\end{equation*}
However the fibers of $E$ over constant loops are canonically trivial,
since these have natural lifts to invertible operators on the Hardy
space, so the action is associative. Thus the conditions in LS.\ref{LSfus}),
excepting lithe smoothness, have been verified.

Next we lift the restricted reparameterization action to $T.$ Thus, suppose
$r\in\Dff_{\{0,\pi\}}(\UU(1))$ lies in the subgroup of leaving both $0$ and
$\pi$ fixed. Thus $r=r_+r_-$ is the composite of diffeomorphisms which fix
one of the two intervals on the circle bounded by these two points. It
therefore suffices to consider just $r=r_+$ since the behaviour of $r_-$ is
essentially the same. This reparameterization of the loop corresponds to
the reparameterization of the path forming the outgoing half of the loop,
with the end points and the incoming half unaffected. The properties of
parallel translation ensure that for a path $p,$ the horizontal path of the
reparameterization is the reparameterization of $h_p$ and the same is then
true of $s_p.$ Thus the blipped outgoing path is
\begin{equation}
\Bl(\gamma\circ r)=(h_p\circ r,s_p\circ r p(2\pi)).
\label{FuLoSptoSt.281}\end{equation}
As for the group action discussed above, this allows the rotation of the
blipped loop, where the return path is unchanged, to be written in terms of
the outer two of the triple of paths
\begin{equation}
(H\circ r',V,V\circ r')\in\cI^{[3]}_{(\pi)}F
\label{FuLoSptoSt.282}\end{equation}
where $r'$ is the reverse of $r.$ The other two fusion loops are then the
rotation of the reparaterization of $\Bl(\gamma)$ by $r$ compressed to a
quarter circle and the rotation of the fiber loop which is given by the loop in $\Spin$
formed from $s_p,$ $s_p\circ r$ and the same return path twice. Using the
fusion isomorphism and rotation invariance of $D,$ the fiber isomorphism to
$E$ and the reparameterization action on $E$ give a natural identification  
\begin{equation}
A(r):T_{\gamma\circ r}\longrightarrow T_\gamma,\ r\in\Dff_{\{0,\pi\}}(\bbS).
\label{FuLoSptoSt.283}\end{equation}
This again can be seen to have the associativity condition required of a
reparameterization action and also to satisfy the compatibility conditions
with fusion for this subgroup and the multiplication action in LS.\ref{LScompat}).

Now the two bundles $T$ and $D$ are, by construction, homotopic through
fusion bundles and hence, following
Proposition~\ref{P:FCB_fusive_isomorphic}, are isomorphic through a
piecewise lithe, fusion preserving and restricted-reparameterization
invariant bundle map. This allows all the structures above to be
transferred back to $D.$ Now, unlike $T,$ $D$ has a full reparameterization
action with the subgroup $\Dff^+_{\{0,\pi\}}(\bbS)$ compatible with the
action $\cL\Spin.$ In fact the smoothness of the action ensures that the
full reparameterization action is compatible, following the arguments of
Lemma~\ref{FuLoSptoSt.456}. The commutativity \eqref{FuLoSptoSt.198} is the
triviality of piecewise lithe function $f:\Rp^+(\bbS)\times\cL\Spin\times\cLE F$ such that 
\begin{equation}
A(r)M(\lambda)l=f(r,\lambda ,l)M(A(r)\lambda)A(r)l,\ l\in\cLE F,\ \lambda
\in\cL\Spin,\ r\in\Rp^+(\bbS).
\label{FuLoSptoSt.489}\end{equation}
Since this is an action, $f(r'r,\lambda ,l)=f(r',A(r)\lambda,A(r)l)f(r,\lambda,l).$
It is shown above that $f(r',\lambda ',l)=1$ for $r'\in\Dff^+_{\{0,\pi\}}(\bbS)$ and
hence for each $r\in\Rp^+(\bbS)$ and $\lambda \in\cL\Spin$ fixed, the
differential of $f(r,\lambda ,l)$ with respect to $r$ vanishes on the Lie
algebra of $\Dff^+_{\{0,\pi\}}(\bbS)$ and hence vanishes identically by the
density with respect to $L^\infty(\bbS).$ 

Finally then it remains to examine the regularity of $T.$ The construction
of the blip curve extends by continuity to the energy space, so the
trivialization of $D$ discussed in the proof of
Theorem~\ref{FuLoSptoSt.168} pulls back to give trivializations of $T$ over
tubular neighborhoods of each loop in the energy topology. The transition
maps for $D$ are lithe in the sense that the derivative at a piecewise
smooth curve, such as $\Bl(\gamma ),$ $\gamma \in\cL F,$ are piecewise
smooth sections of the tangent bundle pulled back to the curve. The
regularity of the blip map is such that the pulled-back transition maps for
$T$ are therefore lithe in the weaker sense that the derivatives are
piecewise smooth over $\gamma$ with the possibility also of delta functions
over $1$ corresponding to the `boundary term' $\nu_F.$ We proceed to show
that this boundary term is in fact absent and that the transition maps have
smooth derivatives, without discontinuities in the derivatives at $0.$
\end{proof}

\appendix
\section{Lithe regularity}\label{Sect.smooth}

The regularity of objects, particulary circle bundles, over the loop space
$\cL M=\CI(\bbS;M)$ of a finite-dimensional oriented, connected and
compact manifold $M$ ultimately reduces to the regularity of
functions. Since loop spaces are modelled on $\CI(\bbS;\bbR^n)$ we first
consider appropriate notions of smoothness for functions on $\CI(Z)$ for a
finite-dimensional manifold $Z$ and then generalize to infinite-dimensional
manifolds modelled on $\CI(Z).$

The Fr\'echet topology on $\CI(Z)$ arises from the standard Sobolev, or
equivalently $\cC^k,$ norms relative to a partition of unity and is given
by the metric
\begin{equation}
d(f,g)=\sum\limits_{i=1}^\infty 2^{-k}\frac{\|f-g\|_k}{1+\|f-g\|_k}.
\label{FuLoSptoSt.67}\end{equation}
By continuity of a function on an open set of $O\subset\CI(Z)$ we will mean
continuity in this sense, so with respect to uniform convergence of all
derivatives.

For such a function the assumption of directional differentiability at each
point leads to a derivative
\begin{equation}
F':O\times\CI(Z)\longrightarrow \bbR.
\label{FuLoSptoSt.179}\end{equation}
Standard notions of continuous differentiability would require this to be
continuous, as a continous linear functional in the second variable.  For
general Fr\'echet spaces it is difficult to refine this, but in the case of
$\CI(Z)$ the pointwise derivative becimes a distributional density on $Z$
and there are many subspaces of distributions. In particular the smooth
densities themselves form a subspace. We define $F$ to be `$\cC^1$-lithe'
if its (weakly defined) derivative arises from a continuous map
\begin{equation}
O\longrightarrow \CI(Z;\Omega )
\label{FuLoSptoSt.180}\end{equation}
through the integration, or `$L^2$', pairing $\CI(Z;\Omega )\times\CI(Z)\longrightarrow
\bbR.$

Higher derivatives, defined by successive weak differentiability, lead to
maps 
\begin{equation}
F^{(k)}:O\times\CI(Z)\times\dots\times\CI(Z)\longrightarrow \bbR
\label{FuLoSptoSt.181}\end{equation}
which are multilinear at each point of $O.$ Given some continuity these
define, by the Schwartz kernel theorem and its extensions, distributional
densities on $Z^k$ symmetric under factor-exchange. One immediate generalization of
\eqref{FuLoSptoSt.180} would be to require that these derivatives factor
through smooth densities on $Z^k$ but this is much too restrictive as 
can be seen in the case $Z=\bbS$ by choosing a smooth function $f\in\CI(\bbR)$ and setting
\begin{equation}
F(u)=\int_{\bbS} f(u(\theta))d\theta,\ u\in\CI(\bbS),
\label{FuLoSptoSt.182}\end{equation}
just the integral of the pull-back. The first derivative at a point $u$ is 
\begin{equation}
F'(u;v)=\int_{\bbS} f'(u(\theta))v(\theta)d\theta
\label{FuLoSptoSt.183}\end{equation}
which is indeed given by pairing with a smooth function on the circle at
each point. However the second derivative is 
\begin{equation}
F''(u;v,w)=\int_{\bbS} f''(u(\theta))v(\theta)w(\theta)d\theta.
\label{FuLoSptoSt.184}\end{equation}
As a distribution on $\bbS^2,$ the torus, this is given by a
distributional density supported on the diagonal, with a smooth
coefficient. It is this property that we generalize.

For any closed embedded submanifold $N\subset Z$ and vector bundle $W$ over
$Z$ the space of `Dirac sections of $W$' with support $N$ is well-defined
by reference to local coordinates and local trivializations. A distribution
supported on $N$ is locally a finite sum of normal derivatives (with
respect to $N)$ of the Dirac delta distribution along $N$ with
distributional coefficients on $N.$ By a Dirac section of order $m$ we mean
that there are at most $m\ge0$ normal derivatives and that all the
coefficients are smooth along $N.$ This space is isomorphic to the space of
smooth sections of a bundle over $N$ (without being one) and this is the
topology we take. There are other somewhat larger spaces one could allow
here corresponding to various classes of distributions conormal with
respect to $N.$

The multidiagonals in $Z^k$ are the the embedded submanifolds which are the
fixed sets of some element of the permutation group acting by
factor-exchange. There are all diffeomorphic to $Z^l$ for some $l\le k.$

\begin{definition}\label{FuLoSptoSt.185} For any compact smooth manifold
$Z$ the space of `Dirac distributional sections' of a vector $W$ over $Z^k$
(for any $k)$ is the direct sum of the Dirac sections of $W,$ as discussed above,
with respect to the multi-diagonals of $Z^k.$ 
\end{definition}

\noindent So for any $k$ this space is topologically the direct sum of copies of
$\CI(Z^j;U_j)$ for $j\le k$ and vector bundles $U_j.$

An important property of these Dirac distributional sections is that they
pull-back under factor exchange to the corresponding space of sections of
the pulled back bundle. Moreover they are preserved under exterior product,
interpreted as multilinear functionals the product of Dirac distributions
on $Z^l$ and $Z^k$ is Dirac on $Z^{l+k}.$ They are also preserved by global
diffeomorphisms of $Z.$

\begin{definition}\label{FuLoSptoSt.186} A \ci-lithe
function on an open subset $O\subset\CI(Z)$ is a function with weak
derivatives of all orders which are given by continuous maps $F^{(k)}$ from
$O$ into the sum of the spaces of Dirac distributional densities on $Z^k,$
supported on the multidiagonals, for each $k.$
\end{definition}
\noindent This definition can be refined by fixing the number of `normal
derivatives' $p_k$ which can appear in the derivatives of order $k.$ We do
not discuss this below, except the important case of $k=1.$ The notion of
`\ci-lithe' is somewhat strengthened in the case of loop manifolds below
by demanding that the functions also be $\cC^1$ on the associated energy
Hilbert manifold.

For a finite-dimensional manifold this notion of \ci-litheness can be construed
as reducing to smoothness in the usual sense. In consequence there is an
immediate extension of the notion of a \ci-lithe function to maps from the
product $U\times O$ of a finite dimensional manifold and an open subset of
$\CI(Z)$ into a finite-dimensional manifold $N$ by simply requiring the
same condition of all derivatives in both variables in coordinate charts on
$N,$ or equivalently considering smooth maps from $U$ into \ci-lithe maps on
$O.$

The \ci-lithe functions on $O$ form a \ci\ algebra. That is, not only are
linear combinations and products of such functions again \ci-lithe, but if
$u_i,$ $i=1,$ $\dots,$ $k,$ are real-valued \ci-lithe functions and $G$ is a
\ci\ function on $\bbR^k$ then $G(u_1,\dots,u_k)$ is also \ci-lithe.

In the body of the paper, heavy use is made of the construction of loops
from paths by fusion. In view of this we need to allow the model space to
be a compact manifold with boundary; in fact we really need (but do not
develop here) as small part of the theory of `articulated manifolds' in
which manifolds with corners are joined to various orders of smoothness
along their boundaries. To keep the discussion within reasonable bounds we
only consider the two cases of a compact manifold with boundary $H$ and
also that of a compact manifold $Z,$ without boundary, but with an interior
dividing hypersurface $H\subset Z.$ In these cases in the products $Z^k$ it
is possible to weaken Definition~\ref{FuLoSptoSt.185}, and hence
Definition~\ref{FuLoSptoSt.186}, by allowing `Dirac sections' over not just
the multidiagonals of $Z^k$ but of submanifolds which are in the same sense
products of $H$ and the various multidiagonals. We distinguish between
these boundary and separating hypersurface cases and weaken
Definition~\ref{FuLoSptoSt.185} as follows.

\begin{definition}\label{FuLoSptoSt.383} In the case that $Z$ is a compact
manifold with boundary the space of Dirac distributional sections of a
vector $W$ over $Z^k$ is the direct sum of the Dirac sections of $W,$ as
discussed above, with respect to the multi-diagonals of $Z^k$ and all their
boundary faces, with smooth coefficients on the as manifolds with
corners. For a compact manifold with interior hypersurface $H$ no terms
supported on factors of $H$ are permitted, and all coefficients are
required to be smooth up to $H$ (separately from both sides). 
\end{definition}
\noindent
Definition~\ref{FuLoSptoSt.186} is then extended by using this notion of
Dirac section.

To extend this definition to a Fr\'echet manifold modelled on $\CI(Z),$ an
appropriate `\ci-lithe' structure is needed. With loop spaces in mind, consider
the open sets in $\CI(Z;\bbR^n)$ which are determined by a corresponding
open subset $Y\subset Z\times \bbR^n$ which is tubular, in the sense that
it fibers over $Z$ and then $O=O_Y\subset\CI(Z;\bbR^n)$ consists of all the
maps with $f(m)\in \{m\}\cap Y$ for all $m,$ i.e.\ all the sections of this
trivial bundle taking values in $Y.$

\begin{proposition}\label{FuLoSptoSt.187} If $Y_i\subset Z\times\bbR^n$ $i=1,2$ are
tubular open subsets and $O_i=\CI(Z;Y_i)\subset\CI(Z;\bbR^n)$ is the
corresponding open subset with values in $Y_i$ then any
fiber-preserving diffeomorphism $T:O_1\longrightarrow O_2$ induces a
bijection between the spaces of \ci-lithe functions on $O_1$ and $O_2.$ 
\end{proposition}

\begin{proof} If $F$ is \ci-lithe on $O_2$ then the pull-back is weakly
differentiable and the derivatives are
\begin{equation}
(T^*F)^{(k)}(u\circ T;v_1,\dots,v_k)=F^{(k)}(u,T_*v_1,\dots,T_*v_k)
\label{FuLoSptoSt.188}\end{equation}
where $T_*$ is the differential of $T$ as a bundle isomorphism along the
section $u\circ T,$ i.e.\ at each point in $O_1$ it is a smooth family,
parameterized by $Z,$ of invertible linear transformations of $\bbR^n.$ The
space of Dirac distributions is preserved by such bundle maps so the
invariance follows.
\end{proof}

\section{Loop manifolds}\label{Sect.Loops}

If $M$ is a finite dimensional oriented compact manifold we consider here
some of the basic properties of the loop space $\cL M=\CI(\bbS;M)$ of $M.$
We also need to consider some related manifolds, in particular the `energy'
loop space $\cLE M=H^1(\bbS;M)$ and the path space $\CI([0,2\pi];M).$ If
$S\subset\bbS$ is finite will also consider the space, $\cL_SM,$ of
continuous loops which are piecewise smooth, i.e.\ smooth on each of the
closed intervals into which $S$ divides $\bbS,$ the case $S=\{0,\pi\}$ is
particularly important. For any $S,$
\begin{equation}
\xymatrix@R1em{
\cL M\ar[dr]\ar[dd]\\
&\cLE M\\
\cL_SM\ar[ur]
}
\label{FuLoSptoSt.384}\end{equation}
are dense subsets of the energy space (but smooth loops are not dense in
piecewise smooth loops in the standard topology).

For a given metric on $M$ let $\epsilon _0$ be the injectivity radius. For
$u\in\cL M$ and $0<\epsilon <\epsilon _0$ the sets
\begin{equation}
\begin{gathered}
\Gamma (u,\epsilon )=
\{v\in\cL M;d_M(v(t),u(t))<\epsilon \ \forall\ t\in\bbS\}\\
\Gamma_S(u,\epsilon )=
\{v\in\cL_S M;d_M(v(t),u(t))<\epsilon \ \forall\ t\in\bbS\}\end{gathered}
\label{FuLoSptoSt.37}\end{equation}
are identified by the exponential map at points along $u$ with a tubular
neighbourhood of the zero section of the pull back to $\bbS$ of the tangent
bundle $TM$ under $u.$ Since this bundle is trivial over $u$ it may be
identified with $\CI(\bbS;\bbR^n)$ or the corresponding piecewise smooth
space $\CI_S(\bbS;\bbR^n).$ These sets form coordinate covers of
$\cL M$ and $\cL_SM$ and for any non-trivial intersection the coordinate transformation
is a loop into the groupoid of local diffeomorphisms, $\Dff(\bbR^n),$ of
$\bbR^n.$ As such it is closely related to the corresponding structure
groupoid of a finite dimensional smooth fiber bundle as noted for example
by Omori \cite{Omori1997}. The convexity properties of small geodesic balls
show that this is actually a good open cover of $\cL M$ as a Fr\'echet
manifold locally modelled on $\cL \bbR^n=\CI(\bbS)^n.$

The sets \eqref{FuLoSptoSt.37} are open with respect to the supremum
topology on continuous loops and since $H^1(\bbS)\subset\cC(\bbS)$ there
are similar open sets in the finite-energy loop space fixed by the $H^1$
norm of the section of the tangent bundle over the base loop representing a
nearby loop
\begin{equation}
\Gamma_E(u,\epsilon')=
\{v\in H^1(\bbS;M);\|v_u(t)\|_{H^1}<\epsilon'\}.
\label{FuLoSptoSt.385}\end{equation}
Since $\cLE M$ is a real Hilbert manifold, the derivative of a function, thought
of as a linear functional on the tangent space which is $H^1(\bbS,u^*TM)$
at $u\in\cLE M,$ would usually be interpreted as acting through the Riesz
identification with the dual. However, in terms of functions on the circle
it is appropriate to interpret the duality through $L^2,$ so here the
derivative $f'(u)$ always acts on the tangent space through 
\begin{equation}
f'(u)(v)=\int_{\bbS}\langle f'(u),v\rangle _{u^*g}d\theta.
\label{FuLoSptoSt.393}\end{equation}
This identifies a bounded linear functional, with respect to the energy
norm, on the tangent space as an element of $H^{-1}(\bbS,u^*TM).$ Thus the
following definition requires enhanced regularity of the derivative even as
a once-differentiable function on $\cLE M.$ 

\begin{definition}\label{FuLoSptoSt.387} By a \emph{lithe} function on the
loop space $\cL M$ is meant a \ci-lithe function in the sense of
Definition~\ref{FuLoSptoSt.186} which has a $\cC^1$ extension to $\cLE M$
with derivative an element of $L^2(\bbS;u^*TM)$ depending continuously on
$u\in\cLE M.$ By an $S$-lithe function, for $S\subset\bbS$ finite, we mean
a \ci-lithe function on $\cL_S M$ (see Defintion~\ref{FuLoSptoSt.383}) with
such a $\cC^1$ extension to $\cLE M.$
\end{definition}

\begin{proposition}\label{FuLoSptoSt.189} 
The condition that a continuous function be lithe on each of the open
subsets $\Gamma(u,\epsilon)$ forming an open cover of $\cL M$ is
independent of the open cover and these functions form a \ci\ algebra on $\cL M.$ 
\end{proposition}

\begin{proof} This is a direct consequence of
Proposition~\ref{FuLoSptoSt.187} above.
\end{proof}

\begin{lemma}\label{FuLoSptoSt.394} The holonomy of a smooth circle bundle
with connection over $M$ is a lithe function on $\cL M$ or any $\cL_SM.$ 
\end{lemma}

\begin{proof} Consider a smooth circle bundle $D,$ with connection, over
$M.$ Over any one loop, $u\in\cL M,$ $D$ is trivial (since the structure
bundle is oriented) i.e.\ has a smooth section. Pulling $D$ back to the
tubular neighborhood $N_\epsilon (u)$ of the zero section of $u^*TM$ via
the expontial map at each point gives a trivial circle bundle to which this
section is extended by parallel transport along the radial paths from each
point $u(t).$ This induces a section of the pull-back of $D$ to each
element of $\Gamma(u,\epsilon )$ which factors through the pull-back to
$N_\epsilon (u).$ The pulled back connection is then represented by a
smooth 1-form on $N_\epsilon(u)$ and the holonomy of $D$ around each loop
$\gamma \in\Gamma(u,\epsilon)$ is given by the integral of this 1-form
pulled back from $N_\epsilon (u)$
\begin{equation}
h(\gamma)=\exp(i\int_{\UU(1)}\alpha(\gamma (\theta ))\cdot\dot\gamma(\theta)d\theta )
\label{FuLoSptoSt.68}\end{equation}
using the pairing of a smooth 1-form pulled back to the curve and the
derivative of the curve. It follows from this that the holonomy is lithe on
each of these open sets and therefore globally on $\cL M.$ Note that the
only derivative within the integral involves the tangent vector field,
$\tau_u(\theta),$ to the curve, used to pull back the 1-form. This is in
$L^2(\bbS,u^*TM)$ and by integration by parts the derivative can always be
thrown onto the background tangent vector so the regularity required in
Definition~\ref{FuLoSptoSt.387} follows.
\end{proof}

The group of oriented diffeomorphisms of the circle, $\Dff^+(\bbS)$ is an
open subset of $\cL\bbS.$ Thus it inherits a lithe structure. It acts on
$\cL M$ by reparameterization and one of the basic ideas in the study of
the loop space is to produce objects which are equivariant with respect to
this action. Observe first that the action, viewed as a map 
\begin{equation}
\Dff^+(\bbS)\times\cL M\longrightarrow \cL M
\label{FuLoSptoSt.190}\end{equation}
is itself lithe, i.e.\ the pull-back of a lithe function under it is lithe.

If $r(s)\in\Dff^+(\bbS)$ is a smooth curve with $r(0)=\Id$ then the
derivative of the pulled back action on lithe functions is readily computed
in terms of the tangent vector field $v=v(\theta)d/d\theta=dr/ds(0) \in \cV(\bbS)$
as 
\begin{equation}
v\cdot f(u)=\frac{d}{ds}\big|_{s=0}f(u\circ r(s))=\int_{\bbS}
v(\theta)\langle f'(u)(\theta),\tau_u(\theta)\rangle _{u^*g}d\theta,\
u\in\cLE M,
\label{FuLoSptoSt.395}\end{equation}
where $\tau_u(\theta)\in L^2(\bbS,u^*TM)$ is again the tangent vector field
to $u.$ 

If $S\subset\bbS$ is finite consider the subgroup
$\Dff^+_S(\bbS)\subset\Dff^+(\bbS)$ of the diffeomorphisms which fix each
point of $S.$ The whole group $\Dff^+(\bbS)$ has the homotopy type of the
circle, in the \ci\ topology while these subgroups are
contractible. Although $\Dff^+_S(\bbS)$ is a closed Fr\'echet subgroup,
with Lie algebra $\cV_S(\bbS)\subset\cV(\bbS)$ consisting of the vector
fields vanishing at $S$ it is, in much a weaker sense, dense in the whole
group. This is one reason to insist on high regularity for functions.

\begin{lemma}\label{FuLoSptoSt.390} If $f:\cLE M\longrightarrow \bbC$ is
$S$-lithe then for each smooth vector $v\in\cV(\bbS)$ and fixed finite set
$S$ there is a sequence $v_n\in\cV_S(\bbS)$ such that 
\begin{equation}
vf(u)=\lim_{n\to\infty} v_nf(u)\ \forall\ u\in\cLE M.
\label{FuLoSptoSt.392}\end{equation}
\end{lemma}

\begin{proof} This follows directly from \eqref{FuLoSptoSt.395} since one
only needs to choose a sequence $v_n\in\cV_S(\bbS)$ bounded in supremum
norm such that $v_n\to v$ almost everywhere and then the integrand in
\eqref{FuLoSptoSt.395} converges in $L^1.$ 
\end{proof}

This leads to a regularity result which is important in the analysis of
circle bundles.

\begin{lemma}\label{FuLoSptoSt.396} If $f$ is an $S$-lithe function on
$O\subset\cL_SM$ for some finite set $S$ and some open set $O$ which is
invariant under $\Dff^+(\bbS)$ and $f$ is invariant under the action of
$\Dff_S(\bbS)$ then it is lithe and invariant under the action of
$\Dff^+(\bbS).$
\end{lemma}

\begin{proof} Invarince under the action of $\Dff^+_S(\bbS)$ implies that
$v\cdot f=0$ on $O.$ Then Lemma~\ref{FuLoSptoSt.390} implies that $v\cdot
f=0$ for all $v\in\cV(\bbS).$ Now, it is not the case these vector fields
exponentiate to a neighborhood of the identity of $\Dff^+(\bbS)$ but any
such diffeomorphism is given by integration of the action of a
parameter-dependent vector field, and hence the triviality of the action
extends to the whole of $\Dff^+(\bbS).$ 
\end{proof}

As well as lithe functions over $\cL M$ we consider circle, i.e.\
principal $\UU(1),$ bundles.

\begin{definition}\label{FuLoSptoSt.192} A circle bundle $D$ over $\cL M$
is $(S)$-lithe if it has trivializations over a covering of $\cL M$ by open sets
\eqref{FuLoSptoSt.37} with lithe transition maps on intersections.
\end{definition}

The inverse of a lithe circle bundle is then lithe, as is the tensor
product of two lithe circle bundles. The notion of a lithe section over an
open set is also well defined. Then a $\UU(1)$ map from one lithe circle
bundle to another will be considered lithe if it corresponds to a lithe
section of the tensor product of the image bundle with the inverse of the
source bundle. Applying Lemma~\ref{FuLoSptoSt.396} to the transition maps
of a lithe $\UU(1)$ bundle proves

\begin{lemma}
A lithe action of $\Dff^+_S(\bbS)$ on a lithe circle bundle $D$ may be extended
to a lithe action of $\Dff^+(\bbS).$
\label{L:reparam_extension_lcb}
\end{lemma}

We also consider the path space in $M,$ $\cI M=\CI([0,\pi];M)$ although the
choice of interval is somewhat arbitrariy. Evaluation at the two end-point
maps gives a fibration fibers over $M^2.$ Indeed, local coordinates based
at a point $\bar m\in U\subset M$ can be used to construct a
family of diffeomorphisms 
\begin{equation}
\begin{gathered}
U\times[0,\pi]\ni (m,t)\longmapsto\psi_{m,s}\in\Dff(M),\\
\psi_{m,s}=\Id\ s>\ha,\ \psi_{m,0}=\Id,\ \psi_{m,1}(\bar m)=m
\end{gathered}.
\label{FuLoSptoSt.40}\end{equation}
and then  
\begin{equation}
u\longrightarrow u'(t)=\psi_{m,t}(u(t))
\label{FuLoSptoSt.41}\end{equation}
is an isomorphism of paths with initial point at $\bar m$ to paths with
initial point at $m$ giving a local section of $\cI M$ with respect to one
end-point and the other end can be treated similarly.

We leave the properties of litheness for functions on the path space as an
exercise, but notice that the appropriate extension of
Definition~\ref{FuLoSptoSt.387} should permit single Dirac delta terms at the
boundary of the interval. For instance, the pull-back to $\cI M$ of a
function on $M$ by mapping to the initial point as derivative given by such
a delta function.

\section{Basic central extension}\label{CentExt}

We consider here properties of the basic central extension of $\cL\Spin,$
the loop group on $\Spin$ in dimensions $n\ge5.$ Waldorf \cite{Waldorf2012}
has shown that it satisfies a form of the fusion condition. Here we show
that the central extension is `fusive' in that it is also lithe and
exhibits reparameterization equivariance. These results are proved using
the realization of the central extension obtained via the Toeplitz algebra.

Consider the Hardy space of smooth functions on the circle with values in
the complexified Clifford algebra; it is the image of the projection $P_H$
which deletes negative Fourier coefficients
\begin{equation*}
H = \{ u \in
\CI(\bbS; \Cl)\;;\; u = \textstyle{\sum_{k \geq 0}} u_k
e^{ik\theta}\big\} \subset \CI\bpns{\bbS; \Cl}.
\label{FuLoSptoSt.397}\end{equation*}
Then $\cL \Spin$ is embedded as a subsspace of the pseudodifferential
operators by compression to $H$ 
\begin{equation*}
\cL\Spin\ni l\longmapsto P_H l P_H\in \Psi^0\bpns{\bbS; \Cl}.
\label{FuLoSptoSt.398}\end{equation*}
This map is injective since principal symbol $\sigma(P_HlP_H)=l$ on the
positive cosphere bundle of the circle. In fact $P_H$ is microlocally equal
to the identity on the positive side and equal to $0$ on the negative side
from which it follows commutators with multiplication operators are
smoothing 
\begin{equation*}
[P_H,l] \in\Psi^{-\infty} = \CI\bpns{\bbS^2; \End(\Cl)},
\label{FuLoSptoSt.399}\end{equation*}
and hence that
\begin{equation}
\Psi_H\bpns{\bbS; \Cl} = P_H\,\bbC\cL \Spin\, P_H + P_H\,
\Psi^{-\infty}\,P_H \overset{\sigma}\longrightarrow \cL \Spin
\label{E:Toeplitz_LSpin} \end{equation} forms
a ring of operators on $H$ with surjective (symbol) homomorphism back to the
group algebra $\cL \Spin.$

Let  
\begin{equation*}
\cG_{H} = \set{B \in \Psi_H\;;\; B^\ast = B^{-1}} \subset\Psi_{H}(\bbS;\Cl)
\label{FuLoSptoSt.402}\end{equation*}
consist of those invertible elements which are unitary with respect to the
$L^2$ inner product on $H.$ It follows from the Toeplitz index theorem and the
simple connectedness of $\Spin$ that any element of $\Psi_H$ with symbol $l \in \cL
\Spin$ is Fredholm with index $0$ (the winding number of $l$), hence has an
invertible perturbation by a smoothing operator. Since the symbol $l$
itself is unitary, the radial part of the polar decomposition of this
invertible operator is of the form $\Id + A$ with $A$ smoothing and it
follows that the lift of $l$ can be deformed to be unitary. Thus 
\begin{equation}
\cG^{-\infty}_{H} \longrightarrow \cG_H \overset{\sigma}\longrightarrow \cL \Spin 
	\label{E:G_infty}
\end{equation}
is a short exact sequence of groups where the kernel is the normal subgroup 
\begin{equation*}
\cG^{-\infty}_{H} = \set{B = \Id_H + A \;;\; B^\ast = B^{-1},\ A \in P_H\, \Psi^{-\infty}\,P_H}
\label{FuLoSptoSt.401}\end{equation*}
consisting of all unitary smoothing perturbations of the identity on $H$;
in particular these differ from $\Id$ by trace class operators on $H.$ 

The Fredholm determinant on $H$
\begin{equation}
\det:\cG^{-\infty}_{H}\longrightarrow \UU(1)
\label{15.2.2013.1}\end{equation}
is a group homomorphism and moreover is invariant under conjugation by the
full group: 
\begin{equation}
\det(UBU^{-1})=\det(B)\ \forall\ U\in\cG_H.
\label{15.2.2013.2}\end{equation}
It follows that $\cG^{-\infty}_{H}/ \cK \cong \UU(1)$, where
\begin{equation}
	\cK = \set{B \in \cG^{-\infty}_{H}\;;\; \det(B) = 1}
	\label{E:determinant_one}
\end{equation}
is also a normal subgroup of $\cG_H$, and therefore passing to quotients in
\eqref{15.2.2013.1} gives a central extension
\begin{equation}
\begin{gathered}
	\UU(1)\longrightarrow E\cL \Spin \overset{\sigma} \longrightarrow \cL \Spin, \\
	E\cL \Spin = \cG_H / \cK 
	\label{E:ELSpin}
\end{gathered}
\end{equation}

Viewed as a principal $\UU(1)$ bundle over $\cL\Spin,$ the group multiplication
on $E = E\cL \Spin$ becomes an isomorphism of $\UU(1)$-principal bundles over
$\cL\Spin^2:$
\begin{equation}
\begin{gathered}
M:\pi_1^*E\otimes\pi_2^*E\longrightarrow m^*E,\\
m:\cL\Spin^2\ni (l_1,l_2)\longmapsto l_1l_2\in\cL\Spin
\end{gathered}
\label{15.2.2013.6}\end{equation}
being the multiplication map on $\cL\Spin$ and $\pi_i$ the two
projections. The group condition then reduces to associativity over
$\cL\Spin^3$ whereby the two maps 
\begin{equation}
\begin{gathered}
\pi_{1}^*E\otimes\pi_{2}^*E\otimes\pi_3^*E\longrightarrow m_3^*E,\\
m_3(l_1,l_2,l_3)=m_2(m_2(l_1,l_2),l_3)=m_2(l_1,m_2(l_2,l_3))=l_1l_2l_3
\end{gathered}
\label{15.2.2013.7}\end{equation}
obtained by applying $M$ first in the left two factor and then again, or in
the right two factors and then again, are the same.

\begin{theorem}
The central extension \eqref{E:ELSpin} is the basic central extension of $\cL
\Spin$, corresponding to the generator of $H^3(\Spin; \bbZ) \cong \bbZ$.
Moreover it is fusive as a circle bundle over $\cL \Spin$, for which the
multiplication map \eqref{15.2.2013.6} is lithe.
\label{T:basic_fusive_extension}
\end{theorem}

The remainder of this section is devoted to a proof of
Theorem~\ref{T:basic_fusive_extension}, which is split into a number of
separate results.

\begin{proposition}\label{FuLoSptoSt.193} 
The basic central extension of $\cL\Spin$ is lithe.
\end{proposition}
\begin{proof} Recall that the loop space of any manifold is covered by
contractible `geodesic tubes' \eqref{FuLoSptoSt.385} for $\epsilon<\epsilon
_0,$ the injectivity radius for a left-invariant metric on $\Spin.$ If
$U_l=P_hlP_H+A$ is a unitary lift  of $l$ into $\cG_{H}$ then for
$0<\epsilon$ sufficiently small 
$P_Hl'P_H+A$ is invertible for all $l'\in \Gamma (l,\epsilon)$ since in
particular the norm on $L^2$ of the difference $P_H(l^{-1}l)P_H-P_H$
vanishes uniformly with $\epsilon.$ Thus replacing $P_Hl'P_H+A$ by the
unitary part of its radial decomposition gives a section of $\cG_{H}$
over $\Gamma (l,\epsilon)$ and hence a trivialization of $E$ there.

The transition map between these trivializations over the intersection
$\Gamma (l_1,\epsilon _1)\cap \Gamma (l_2,\epsilon _2)$ is given by the
circular part of the determinant of the relative factors
\begin{equation}
\det\left((P_HlP_H+A_1)^{-1}(P_HlP_H+A_2)\right)=
\det\left(P_H+(P_HlP_H+A_1)^{-1}(A_2-A_1)\right).
\label{FuLoSptoSt.145}\end{equation}
Using the standard formula for the derivative of the (entire) function
given by the determinant 
\begin{equation}
d_B\det(P_H+B)=\det(P_H+B)\Tr
\label{FuLoSptoSt.146}\end{equation}
as a linear functional on the Toeplitz smoothing operators at $B,$ it
follows that \eqref{FuLoSptoSt.146} is infinitely differentiable with high
derivatives coming from repeated differentiation of the determinant and of
the inverse in \eqref{FuLoSptoSt.145}. Thus as a polynomial on the k-fold
tensor product of the tangent space $\cL\spin$ to $\cL\Spin,$ this is the
symmetrization of a sum of products of terms each of which is the trace of
a repeated product
\begin{multline}
\Tr\big((P_HlP_H+A_1)^{-1}L_1(P_HlP_H+A_1)^{-1}(A_2-A_1))\cdots\\
(P_HlP_H+A_1)^{-1}L_j(P_HlP_H+A_1)^{-1}(A_2-A_1))\big)
\label{FuLoSptoSt.147}\end{multline}
where the $L_i\in\cL\spin$ are the tangent vectors. These are of
the form as required in the `lithe' condition on functions in
\S\ref{Sect.smooth}, as are the imaginary parts, so $E$ is a lithe circle
bundle. 
\end{proof}

To classify the central extension as an element of $H^3(\Spin; \bbZ) \cong
H^2(\cL \Spin; \bbZ)$, we construct a natural connection on $E\cL \Spin$ as a
principal $\UU(1)$-bundle. First consider the Maurer-Cartan form 
\[
	g^{-1}\,dg \in \Omega^1(\cG_H; T_\Id \cG_H),
\]
which is well-defined as $\cG_H$ is an open subset of the algebra $\Psi_H.$
Here $dg$ takes values in the Lie algebra $T_\Id \cG_H = \Psi_H.$
Restricted to $\cG^{-\infty}_{H}$, for which the Lie algebra $T_\Id
\cG^{-\infty}_{H} = P_H\,\Psi^{-\infty}\,P_H$ consists of trace-class
operators on $H$, the trace
\[
	\Tr(g^{-1}\,dg) \in \Omega^1(\cG^{-\infty}_{H}; \bbC)
\]
is well-defined, and from the fact that the trace is the logarithmic derivative
of the determinant it follows that $\Tr(g^{-1}\,dg)$ vanishes on $\cK = \set{B \in
\cG^{-\infty}_{H}\;;\; \det(B) = 1}$, and hence descends to the
Maurer-Cartan form on
\[
	\UU(1) \cong \cG^{-\infty}_{H} / \cK.
\]
To extend the form $\Tr(g^{-1}\,dg)$ to the whole group $\cG_H$, we use the
regularized trace and its relation to the residue trace of Wodzicki
\cite{wodzicki1987noncommutative} and Guillemin
\cite{guillemin1993residue}. For $A \in \Psi^\bbZ(\bbS; \Cl)$
the regularized trace is defined to be
\[
\begin{gathered}
\ol \Tr(A) = \lim_{z\to 0}
\left( \Tr\bpns{(1 + D_\theta^2)^{-z/2} A} - \tfrac 1 z \Tr_R(A)\right), \\
\Tr_R(A) = \lim_{z \to 0} z \Tr\bpns{(1 + D_\theta^2)^{-z/2} A}.
\end{gathered}
\]
Here $\Tr_R$ is the residue trace, which vanishes on trace-class
operators and satisfies $\Tr_R([A,B]) = 0,$ whereas the regularized trace $\ol
\Tr$ extends the trace from trace-class operators, but is not a trace,
rather satisfies the trace defect formula 
\begin{equation}
	\ol \Tr([A,B]) = \Tr_R\bpns{[B,\log((1 + D_\theta^2)^{-\ha})]\,A}.
\label{FuLoSptoSt.481}\end{equation}

\begin{proposition}(See also \cite{FIOpaper})
The one form $\ol \Tr(g^{-1}\,dg) \in \Omega^1(\cG_H; \bbC)$ descends to a
connection form 
\begin{equation}
	\omega = [\ol \Tr(g^{-1}\,dg)] \in \Omega^1(E\cL \Spin; \bbC)
\label{FuLoSptoSt.480}\end{equation}
which has curvature
\begin{equation}
d\omega(a,b) = -\frac{1}{2\pi i} \int_{\bbS} \tr\bpns{a(\theta)\,b'(\theta)}\,d\theta \
a,\ b \in \cL \mathfrak{spin}.
	\label{E:basic_2cycle}
\end{equation}
\label{P:basic_central_extension}
\end{proposition}
\noindent In particular $[E]\in H^2(\cL \Spin ;\bbZ) \cong \bbZ,$
is a generator and hence the central extension $E\cL \Spin$ is basic.

\begin{proof}
It follows from the above discussion above that $\ol\Tr(g^{-1}\,dg)$ is an
extension of $\Tr(g^{-1}\,dg)$ from $\cG^{-\infty}_{H}$ to $\cG_H$, and
descends to a 1-form on $E\cL \Spin$ whose restriction to the fibers is the
Maurer-Cartan form on $\UU(1).$ The invariance of $g^{-1}dg$ also implies
the invariance of this form the left action of an element of$E \cL\Spin.$
Equivariance with respect to the action by $\UU(1)$ follows from the
invariance of the regularizing operator $(1+D^2_\theta)^{-z/2}$ under rotation and
the invariance of the (analytic continuation) of the trace, so $\omega$ is
indeed a connection on $E\cL \Spin.$

Now $d\ol\Tr(g^{-1}\,dg)=-\Tr(g^{-1}\,dg g^{-1}\,dg)$ so at the identity,
using \eqref{FuLoSptoSt.481}
\begin{equation}
	d\omega(a,b) = -\ol \Tr([A,B])
=-\Tr_R([B,\log(1+D^2_\theta)^{-z/2}]\,A),\ \forall\ a,\ b \in \cL \mathfrak{spin}
	\label{E:trace_commutator}
\end{equation}
here $A,$ $B \in \Psi_H$ are any
self-adjoint lifts of $a,$ $b;$ we take $A = P_H a P_H$ and $B = P_H b P_H.$ It suffices to 
compute \eqref{E:trace_commutator} in terms of the coordinate $\theta \in
(0,2\pi)$ with cotangent variable $\xi = 
(d\theta)^\ast$. With this convention the symbol of the regularizer $(1 +
D_\theta^2)^{-\ha}$ is the boundary defining function $\rho = (1 +
\xi^2)^{-\ha}$ for the radial compactification of $T^\ast \bbS$. 

The residue trace of any element $P_HPP_H \in \Psi^k(\bbS),$ $k \in \bbZ$ is
\[
\Tr_R(P_HPP_H) = \frac{1}{2\pi} \int_{\bbS^+}\tr \sigma_{-1}(P)
\]
where $\sigma_{-1}(P)$ denotes coefficient of $\rho^1$ in the expansion of
the full symbol. Since $P=[B,\log((1 + D_\theta^2)^{-1/2})] \in \Psi^{-1}$
only the principal symbol 
\begin{equation*}
\sigma_{-1}([B,\log((1+D_\theta^2)^{-1/2})]) =
- i \pa_\theta \sigma(B)(\theta)
\end{equation*}
is involved, so \eqref{E:basic_2cycle} follows. This is 2-cocycle generating $H^2(\cL
\Spin; \bbZ)$ -- see for instance \cite{Pressley-Segal1}. 
\end{proof}

\begin{lemma}\label{L:central_extension_equivariance}
As a $\UU(1)$-bundle over $\cL \Spin$, $E\cL \Spin$ has an equivariant
action of $\Dff^+(\bbS)$
\begin{equation*}
A : a^\ast E \longrightarrow \pi_2^\ast E, \
	a : \Dff^+(\bbS) \times \cL \Spin \longrightarrow \cL \Spin
\label{FuLoSptoSt.403}\end{equation*}
under which the connection \eqref{FuLoSptoSt.480} is invariant.
\end{lemma}

\begin{proof} The group $\Dff^+(\bbS)$ acts on $\CI(\bbS); \Cl)$ by
pullback, though not preserving $H$ in general nor is the action
unitary. To make it unitary we use the action on half-densities, rather
than introduce them formally this simply means twisting the action on
sections of the Clifford bundle to
\begin{equation}
F^{\#}u(\theta)=u(F(\theta ))|F'(\theta)|^{\ha}
\label{FuLoSptoSt.482}\end{equation}
where of course $f'>0$ on $\Dff^+(\bbS).$ Nontheless, for any
$F\in\Dff^+(\bbS)$, $P_H\,F^{\#} P_H$ is Fredholm with index $0$ since
$\Dff^+(\bbS)$ is connected to the identity and $P_H F^{\#} (\Id - P_H)$
and $(\Id - P_H)F^{\#} P_H$ are smoothing operators. Thus $P_H\,F^{\#} P_H$
can be perturbed by a smoothing operator to a unitary operator, $U_F,$ acting on $H.$

Now, the action of a diffeomorphism $F\in\Dff^+(\bbS)$ on $\cL\Spin$ is
by composition and it follows that if $L\in\cG_{H}$ is a lift
of $l\in \cL\Spin,$ so $\sigma (L)=l,$ then 
\begin{equation}
\sigma (U_FLU_F^{-1})=l\circ F.
\label{15.2.2013.9}\end{equation}
Changing $U_F$ to another unitary extension $U'_F=BU_F,$ $B\in
\cG^{-\infty}_{H}$ changes the lift of $l\circ F$ to 
\begin{equation*}
U'_F L {U'_F}^{-1} = \left(BGB^{-1}G^{-1}\right)U_FLU_F^{-1},\ G=U_FLU_F^{-1},
\label{15.2.2013.10}\end{equation*}
and $BGB^{-1}G^{-1}\in \cK$ since it has determinant $1.$ It follows that there is
a natural isomorphism of $E$ covering the action of $F:$ 
\begin{equation}
\tilde F:F^*E\longrightarrow E
\label{15.2.2013.11}\end{equation}
which commutes with the group action, and is itself multiplicative, giving
an action of $\Dff^+(\bbS).$

To see that this action on $E$ leaves the connection form
\eqref{FuLoSptoSt.481} it suffices to examine the infinitesmal
action. Since the span the Lie algebra, it is enough to consider vector
fields on the circle $v(\theta)d/d\theta$ with $0<v\in\CI(\bbS).$ The
infinitesmal generator of the half-density action \eqref{FuLoSptoSt.482} is
the selfadjoint first order differential operator
\begin{equation}
V=v^{\ha}D_\theta v^{\ha}.
\label{FuLoSptoSt.483}\end{equation}
For small $s$, $P_H \exp(is V) P_H$ is invertible and unitary,
and the variation 
\begin{equation*}
\frac{d}{ds}\big|_{s=0}\ol \Tr\big(P_H\exp(-isV)P_Hg^{-1}\,dgP_H\exp(isV)P_H\big) = \ol \Tr([g^{-1}\,dg,iV])
\label{FuLoSptoSt.484}\end{equation*}
is given by residue trace 
\[
\Tr_R([iV,\log((1 + D_\theta^2)^{-\ha})]\, g^{-1}\,dg).
\]
Since $V$ is the Weyl quatization of the symbol $v(\theta)\xi,$ the 
symbolic expansion of the commutator has no term of order $-1$, and this residue
trace vanishes.
\end{proof}

Note that the construction of $E\cL \Spin$ could as well have been
carried out using any subspace of $\CI(\bbS; \Cl)$ with projection
operator differing from $P_H$ by a smoothing operator. In particular this
applies to the shifted Hardy spaces 
\begin{equation*}
H_k = \big\{u \in\CI\;;\; u = \textstyle{\sum_{n \geq k}} u_n e^{in \theta}\big\},\ k \in \bbZ
\label{FuLoSptoSt.404}\end{equation*}
resulting in bundles $E_k.$ 
There are explicit isomorphisms $\cG_{H_k} \longrightarrow \cG_{H_{k'}}$
given by conjugation by the unitary multiplication operators $e^{i(k'-k)\theta}$ and the
identity $e^{i (k'-k) \theta} P_{H_k} = P_{H_{k'}}e^{i(k'-k) \theta}$
descend to natural isomorphisms 
\begin{equation}
	E_{k} \cL \Spin \cong E_{k'} \cL \Spin, \quad \forall\ k \in \bbZ.
	\label{E:ce_ind_shift}
\end{equation}

On the other hand, the use of a complementary subspace leads naturally to the
opposite central extension. This can either be seen by going through the above
construction and noting that the integral in the proof of
Proposition~\ref{P:basic_central_extension} is over $\bbS^-$ instead of
$\bbS^+$, hence has the opposite orientation, or by the following argument.
Denoting the complementary subspace to $H_k$ by 
\[
	\ol H_k = \big\{u \in \CI\;;\; u = \textstyle{\sum_{n < k}} u_n e^{in \theta}\big\},
\]
consider the direct sum algebra $\Psi_{H_k} \oplus \Psi_{\ol H_k}$, with the
unitary subgroup 
\[
	\cG_{H_k} \times_{\cL \Spin} \cG_{\ol H_k} = 
	\set{(B,\ol B) \in \cG_{H_k}\times \cG_{\ol H_k}\;;\; \sigma(B) = \sigma(\ol B) \in \cL \Spin}.
\]
The kernel of the symbol map is $\cG^{-\infty}_{H_k}\times \cG^{-\infty}_{\ol
H_k}$, and taking the quotient by the kernel of $\det_{H_k}\times \det_{\ol
H_k}$ leads to the product $E_{k} \times \ol E_{k}$ as a $\UU(1)\times \UU(1)$
bundle, while taking the quotient by the kernel of the full Fredholm
determinant leads to the tensor product
\[
	E_k\cL \Spin\otimes \ol E_{k}\cL \Spin \to \cL \Spin
\]
as $\UU(1)$-bundles over $\cL\Spin.$
However this has a canonical section, given by the fact that any $l \in \cL
\Spin$ is already a unitary operator on the space $H_k \oplus \ol H_k =
\CI$, which leads to a canonical isomorphism 
\begin{equation}
	\ol E_{k}\cL \Spin \cong (E_k\cL \Spin)^{-1} \quad \forall\, k \in \bbZ.
	\label{E:ce_complement_inverse}
\end{equation}

\begin{lemma}
$E\cL \Spin \to \cL \Spin$ is equivariant with respect to the full
diffeomorphism group $\Dff(\bbS)$, that is
\[
\begin{gathered}
	A : a^\ast E \overset{\cong}\longrightarrow \pi_2^\ast E^{\pm 1}, \\
	a : \Dff(\bbS)\times \cL \Spin \to \cL\Spin
\end{gathered}
\]
with sign $-1$ over the component $\Dff^-(\bbS)\times \cL \Spin$ and $+1$
over $\Dff^+(\bbS)\times \cL \Spin.$
\label{L:full_equivariance}
\end{lemma}
\begin{proof}
In light of Lemma~\ref{L:central_extension_equivariance} it suffices to check
that $\alpha^\ast E\cong E^{-1}$ with respect to any one $\alpha \in
\Dff^-(\bbS)$ which reverses orientation, an obvious choice being the
involution $\alpha(\theta) = 2\pi-\theta.$ Observe that pull-back by $\alpha$
interchanges $e^{ik\theta}$ with $e^{-ik \theta}$ and hence $\alpha^\ast : H_0
\to \ol H_1$ is a unitary isomorphism; here $\ol H_1 = \CI \cap
\text{span}\set{e^{ik\theta} \;;\; k \leq 0}$.

If $L \in \cG_H$ is a unitary lift of $l\circ \alpha \in \cL \Spin$, it follows
from the above discussion that $\alpha^\ast L \alpha^\ast \in \cG_{\ol H_1}$ is a
unitary operator on $\ol H_1$, with symbol $l\circ \alpha^2 = l$. Using
\eqref{E:ce_complement_inverse} and \eqref{E:ce_ind_shift} leads to the
canonical isomorphism
\[
	\alpha^\ast E \overset \cong \longrightarrow E^{-1}
\]
which concludes the proof.
\end{proof}

The fusion property now follows from equivariance and the multiplicativity \eqref{15.2.2013.6}.
\begin{corollary}
$E\cL \Spin \to \cL \Spin$ has the fusion property as a $\UU(1)$-bundle.
\label{C:ce_fusion}
\end{corollary}
\begin{proof}
For any $l \in \cL \Spin$ which is invariant under the action of the loop
reversal involution $\alpha,$
\[
	l\circ \alpha = l \implies E_l \cong \UU(1)	
\]
since the fiber $E_l$ is isomorphic to its inverse $E_l^{-1}.$ In particular
$E$ is naturally trivial, in a way that is consistent with the group multiplication,
over `there-and-back' paths
\[
	E_{\psi(\gamma,\gamma)} \equiv \UU(1)\ \forall\ \gamma \in \cI \Spin.
\]

This in turn leads to the full fusion condition for $E.$ Namely if
$(\gamma_1,\gamma_2,\gamma_3) \in \cI^{[3]}\Spin$, 
then for the three loops $l_{12} = \psi(\gamma_1,\gamma_2)$, 
$l_{23}=\psi(\gamma_2,\gamma_3)$ and $l_{23}=\psi(\gamma_1,\gamma_3),$ then
\[
	l_{12}l_{13}^{-1}l_{23}=\psi(\gamma_2,\gamma_2) \in \cL \Spin
\]
from which it follows that
\begin{equation}
	E_{l_{12}}\otimes E_{l_{23}}\simeq E_{l_{13}}. \qedhere
\label{FuLoSptoSt.139}
\end{equation}
\end{proof}

\bibliography{FuLoSptoSt}
\bibliographystyle{amsplain}

\end{document}